\documentclass[10pt]{amsart}
\usepackage{amscd, amssymb, mathrsfs, mathabx, tikz-cd, amsthm, thmtools}

\author[Kim]{Bumsig Kim}
\address{Korea Institute for Advanced Study\\
85 Hoegiro \\
Dongdaemun-gu\\
Seoul 02455\\
Republic of Korea}
\email{bumsig@kias.re.kr }

\author[Oh]{Jeongseok Oh}
\address{
Department of Mathematics \\
Imperial College \\
London SW7 2AZ \\
United Kingdom
}
\email{j.oh@imperial.ac.uk}

\usepackage{hyperref}

\input{xy}
\xyoption{all}

\newtheorem{Thm}{Theorem}[section]
\newtheorem{Conj}[Thm]{Conjecture}

\newtheorem{Def/Thm}[Thm]{Definition/Theorem}
\newtheorem{Cor}[Thm]{Corollary}
\newtheorem{Lemma}[Thm]{Lemma}

\theoremstyle{definition}
\newtheorem{Rmk}[Thm]{Remark}
\newtheorem{example}{Example}

\numberwithin{equation}{section}
\newcommand{\ti }{\times}
\newcommand{\ot }{\otimes}
\newcommand{\ra }{\rightarrow}

\newcommand{\Hom }{{\mathrm{Hom}}}

\newcommand{\Spec}{{\mathrm{Spec}}}

\newcommand{\Res}{{\mathrm{Res}}}
\newcommand{\Ker}{{\mathrm{Ker}}}

\newcommand{\Sym}{{\mathrm{Sym}}}

\newcommand{\rank }{{\mathrm{rank}}}
\newcommand{\cA}{{\mathcal{A}}}
\newcommand{\cO}{{\mathcal{O}}}

\newcommand{\cE}{{\mathcal{E}}}
\newcommand{\cF}{{\mathcal{F}}}
\newcommand{\cH}{{\mathcal{H}}}
\newcommand{\cP}{{\mathcal{P}}}

\newcommand{\cV}{{\mathcal{V}}}

\newcommand{\cT}{{\mathcal{T}}}

\newcommand{\fC}{{\mathfrak{C}}}

\newcommand{\G}{{\bf G}}

\newcommand{\PP }{{\mathbb P}}
\newcommand{\GG }{{\mathbb G}}
\newcommand{\QQ }{{\mathbb Q}}
\newcommand{\CC }{{\mathbb C}}
\newcommand{\ZZ }{{\mathbb Z}}
\newcommand{\RR }{{\mathbb R}}

\newcommand{\ke }{{\varepsilon }}
\newcommand{\kb }{{\beta}}
\newcommand{\ka }{{\alpha}}

\newcommand{\ch}{\mathrm{ch}}

\newcommand{\fM}{\mathfrak{M}}

\newcommand{\lan}{\langle}
\newcommand{\ran}{\rangle}

\newcommand{\uf}{\underline{f}}

\newcommand{\pr}{\mathrm{pr}}

\newcommand{\fB}{\mathfrak{B}}

\newcommand{\vir}{\mathrm{vir}}

\newcommand{\td}{\mathrm{td}}

\newcommand{\kl}{\lambda}

\newcommand{\cB}{\mathcal{B}}

\newcommand{\spp}{\mathrm{Sp}}
\newcommand{\Zs}{Z(s)}
\newcommand{\fBo}{\fB ^{\circ}}
\newcommand{\RRes}{\RR\mathrm{es} }
\newcommand{\Zb}{Z(\beta)}

\newcommand{\LGQ}{LGQ}

\newcommand{\Coker}{\mathrm{Coker}}
\newcommand{\tot}{\mathrm{tot}}
\newcommand{\dwq}{dw_{\LGQ '}}
\newcommand{\UPQ}{U_{PQ}}

 \newcommand{\bfA}{{\mathbf A}}
 \newcommand{\bfB}{{\mathbf B}}

\makeatletter
\@namedef{subjclassname@2020}{%
  \textup{2020} Mathematics Subject Classification}
\makeatother

\begin{document}

\title{Localized Chern Characters for 2-periodic complexes}

\begin{abstract}
For a two-periodic complex of vector bundles, Polishchuk and Vaintrob have
constructed its localized Chern character. 
We explore some basic properties of this localized Chern character. 
In particular, we show that the cosection localization defined by Kiem and Li is equivalent to
a localized Chern character operation for the associated two-periodic Koszul complex, strengthening 
a work of Chang, Li, and Li. We apply this equivalence to the comparison of virtual classes
of moduli of $\ke$-stable quasimaps and moduli of the corresponding LG $\ke$-stable quasimaps, in full generality. 
\end{abstract}

\keywords{two periodic complexes, localized Chern characters, cosection localizations, virtual classes, LG (quasi)maps with $p$-fields}

\subjclass[2020]{14N35(primary), and 53D45(secondary)}

\maketitle

\setcounter{tocdepth}{1}
\tableofcontents


\section{Introduction}
\subsection{Main Results}
Let $Y$ be a finite type Deligne-Mumford stack over a fixed base field $\bf k$ and let
 $X\xrightarrow{i}Y$ be the inclusion of a closed substack $X$ of $Y$.
Let $E^{\bullet}$ be a 2-periodic complex
of vector bundles, which is exact off $X$:
 $$[ \xymatrix{   E^{-}  \ar@/^.2pc/[r]^{d^{-} }  
&  \ar@/^.2pc/[l]^{d^{+} } E^{+}  }] = ... \xrightarrow{d^+} E^{-} \xrightarrow{d^{-}} E^+ \xrightarrow{d^+} E^{-} \xrightarrow{d^{-}}  ... $$ 
$E^+$ is in even degree and $E^-$ is in odd degree.
 Suppose that $\Ker\, d^{-}$ and $\Ker\, d^{+}$ restricted to $Y - X$
  are vector bundles.

In paper \cite{PV:A} Polishchuk and Vaintrob   define a bivariant class 
$$\ch ^Y_X (E^{\bullet}) \in A^*(X \xrightarrow{i} Y)_{\QQ}$$ generalizing the localized Chern characters constructed by
Baum, Fulton, and MacPherson  \cite{BFM}. For each $[V]\in A_*(Y)_{\QQ}$, this  assigns a class
\[ \ch ^Y_X (E^{\bullet}) \cap [V] \in A_*(X)_{\QQ} \] 
whose image in $A_*(Y)_{\QQ}$ coincides with $\ch(E^+) \cap [V] - \ch (E^-)\cap [V]$.

\medskip

Polishchuk and Vaintrob \cite{PV:A} use the generalized localized Chern characters to define 
Witten's top Chern class. This is a particular case of pure Landau-Ginzburg phases in gauged linear sigma model. H.-L. Chang, J. Li and W.-P. Li also define Witten's top Chern class via cosection 
localization. They show that both constructions coincide; see \cite[Proposition 5.10]{CLL}.
This is a special case of the equivalence that a cosection localization  
of Kiem-Li \cite{KiemLi} is the localized Chern character for the associated 2-periodic Koszul complex.
We prove the following equivalence.   
  
Let $p: F\to M$ be a vector bundle on a DM stack $M$ 
and consider a cosection $\sigma \in H^0(M, F^{\vee})$ of $F$. It 
induces a function $w_{\sigma} : F \to \mathbb{A}^1$. Denote by $Z(w_{\sigma}) \subset F$ and $Z(\sigma) \subset M$
the zero loci of $w_{\sigma}$ and $\sigma$, respectively. 

\begin{Thm}\label{Thm1} Let  $0^!_{F, \sigma }$ denote the cosection localization in $A_*(Z(\sigma) \to Z(w_{\sigma}))_{\QQ}$ and let
$\{p^*\sigma , t_F\}$ be the Koszul complex of the pair of cosection $p^*\sigma$ and the tautological section $t_F$ of $p^*F$. Then
\[ 0^!_{F, \sigma } = \td F |_{Z(\sigma )}  \cdot \ch ^{Z(w_{\sigma})}_{Z(\sigma)} (\{p^*\sigma , t_F\}) . \]
\end{Thm}

 Also by this approach 
 we may define the virtual structure sheaves  and study the comparisons of those defined by \cite{YPLee} and \cite{KiemLi2}, respectively. 
 This is treated in \cite{OS}.

\medskip
We apply Theorem \ref{Thm1} to the comparisons  
of the following virtual classes. 

Let $V_1$ be a vector space with the standard diagonal action by the multiplicative group $\mathbb{G}_m$ so that $\PP V_1 = [V_1 - \{ 0\} /\mathbb{G}_m]$, the space 
of 1-dimensional subspaces of $V_1$.
Let $V_2$ be a $\mathbb{G}_m$-space, a vector space with a linear action by $\mathbb{G}_m$. 
Consider a $\mathbb{G}_m$-invariant element  $w$ of $(\Sym ^{\bullet} V_1^\vee) \ot V_2^\vee$.
Let $E = [(V_1 - \{ 0 \} \times V_2) / \mathbb{G}_m] $, which is 
 a vector bundle on $\PP V_1$. Then $E$ has a cosection associated to $w$.
This cosection amounts to a function $\underline{w} : E \to \mathbb{A}^1$ which is linear in
fiber coordinates of $E$. 

In paper \cite{CL} H.-L. Chang and J. Li  introduce a moduli space $LGQ^{\infty} _{g} (E,  d)' $ of unpointed genus $g$, degree $d$, stable maps to a complex projective space
$\PP V_1$ with $p$-fields and construct a cosection $dw_{LGQ'}$ of the obstruction sheaf 
and a virtual class $[ LGQ^{\infty}_{g} (E,  d)']^{\vir}_{dw_{LGQ'}}$ 
via cosection localization. 
This is a particular case of geometric phases in gauged linear sigma model. Let $Z(d\underline{w}) \subset E$ denote the critical locus of $\underline{w}$.
When $E$ is the line bundle $\mathcal{O}_{\PP ^4}(-5)$ with $Z(d\underline{w})$ a smooth quintic hypersurface, 
Chang and Li show that for $d\ne 0$,  the degree of $[ LGQ^{\infty}_{g} (E,  d)']^{\vir}_{dw_{LGQ'}}$ 
coincides with, up to an explicit sign, 
the degree of the virtual class $[Q^{\infty}_{g} (Z(d\underline{w}), d)]^{\vir}$ of the moduli space
 $Q^{\infty}_{g} (Z(d\underline{w}), d)$ of unpointed genus $g$, degree $d$, stable maps to the quintic.
We prove the following generalization of it.

For any geometric gauged linear sigma model $(V=V_1\oplus V_2, G, w)$ (see \S 3.1)
and any $\ke \in \QQ _{>0}$, we have
the cosection localized virtual class $[LGQ^{\ke}_{g, k} (E, d)']^{\vir}_{dw_{LGQ'}}$
of the moduli space $LGQ^{\ke}_{g, k} (E,  d)'$ of $\ke$-stable quasimaps to $V_1/\!\!/G$ with $p$-fields 
and the virtual  class  $[Q^{\ke}_{g, k} ( Z(d\underline{w}) , d)]^{\vir}$  of the moduli space $Q^{\ke}_{g, k} (Z(d\underline{w}), d)$ of $\ke$-stable quasimaps to $Z(d\underline{w})$.

\begin{Thm}\label{Comp:Thm} 
In $A_*(Q^{\ke}_{g, k} ( Z(d\underline{w}) , d))_{\QQ}$,
\begin{align*}\label{Vir: Eq} [Q^{\ke}_{g, k} ( Z(d\underline{w}) , d)]^{\vir} = (-1)^{\chi ( \cV ^{\vee}_2 ) } [LGQ^{\ke}_{g, k} (E, d)']^{\vir}_{dw_{LGQ'}} 
\end{align*}
where $\chi ( \cV ^{\vee}_2 )$ is the virtual rank of  the complex on  $LGQ^{\ke}_{g, k} (E, d)'$ induced from $V_2$ (see Conjecture \ref{conj}). 
\end{Thm}


\subsection{Acknowledgments}  B. Kim would like to thank Yongbin Ruan for drawing his attention to the comparison question of 
virtual classes, Andrei Okounkov for stimulating comments, and Arkady Vaintrob for answering a question.
The authors would like to thank Ionu\c t Ciocan-Fontanine, Tom Graber and Taejung Kim for helpful comments in shaping the paper.
This material is based upon work supported by NSF grant DMS-1440140 while the first author was in residence at
MSRI in Berkeley during Spring 2018 semester.
J. Oh would like to thank Sanghyeon Lee for useful discussions and University of California, Berkeley for excellent working conditions.
B. Kim is partially supported by KIAS individual grant MG016404.
J. Oh is partially supported by KIAS individual grant MG063002.


\subsection{Conventions} 
By a vector bundle $E$ on an Artin stack $Y$, we will mean a locally free coherent sheaf on $Y$.
We will also call its total space $\tot E$ the vector bundle and often denote it by $E$ if there is no danger to be misunderstood. 
For a morphism $f: X\to Y$ between Artin stacks, $E|_X$ denotes the pullback vector bundle $f^*E$.

\section{Localized Chern Characters of Koszul Complexes}

In this section, we briefly recall the definition of localized Chern characters for 2-periodic complexes and
introduce tautological Koszul complexes attached to Koszul 2-periodic complexes. Then we will show that the cosection localization
coincides, up to the Todd factor, with the localized Chern character for the associated tautological Koszul complex (i.e., Theorem \ref{Thm1}).

\subsection{Definition}
By a {\em 2-period complex  of vector bundles} on $Y$, we mean a $\ZZ/2$-graded vector bundle $E^{\bullet} = E^+\oplus E^-$ on $Y$ with
an odd degree vector bundle map $d_E : E^{\bullet} \to E^{\bullet}$ such that $d_E^2=0$.  Here 
the even degree part of $E^\bullet$ is denoted by $E^+$ and the odd degree part of $E^{\bullet}$ is denoted by $E^{-}$.
We write $d^{\pm}_E := d_E |_{E^{\pm}} : E^{\pm} \to E^{\mp}$. 
When it is clear, we suppress the subscript $E$ writing simply $d$, $d^{\pm}$ for $d_E$, $d^{\pm}_E$, respectively.
 A morphism $f: E^{\bullet} \to F^{\bullet}$ from $E^{\bullet}$ to another  2-periodic complex $F^{\bullet}$ is a degree preserving $\cO_Y$-module homomorphism such that
$f \circ d_E = d_F \circ f$. We will say that $f$ is an {\em isomorphism} if there is another morphism $g: F^{\bullet} \to E^{\bullet}$ such that
$f \circ g = \mathrm{id}_F$, $g\circ f = \mathrm{id}_E$.

Let $X$ be a closed substack of $Y$ and denote by  $i: X \to Y$ the immersion.
We call $E$ {\em strictly exact off} $X$ if $E^{\bullet}$ is exact off $X$
and both sheaves $\Ker\, d^{\pm}$ restricted to $Y - X$ are locally-free coherent sheaves.
Here by that $E^{\bullet}$ is exact off $X$, we mean that  the natural maps $\mathrm{Im} (d^{\mp}|_{Y-X}) \to \Ker (d^{\pm}|_{Y-X} )$ 
are isomorphisms.

We recall the definition of localized Chern characters 
\[ \ch^Y_X (E^{\bullet}) \in A^* (X\xrightarrow{i} Y)_{\QQ} \]
for a 2-period complex $E^{\bullet}$ of vector bundles which is strictly exact off $X$; see \cite{PV:A}.
We refer Chapter 17 of \cite{Ful} for the definition of the group $A^*(X\xrightarrow{i} Y)_{\QQ}$ of bivariant classes for the map $i$. 
For each morphism $g: Y' \to Y$ from a DM stack $Y'$ to $Y$, we need to define 
$$\ch^Y_X (E^{\bullet})_g : A_*(Y')_{\QQ} \to A_*(X')_{\QQ}, \  \gamma \to \ch^Y_X (E^{\bullet})_g \cap \gamma $$
where $X':= X \ti _Y Y'$. If understood, we will drop the subscript $g$ in the notation $\ch^Y_X (E^{\bullet})_g$.

First consider the case when $g=\mathrm{id}$.
For a cycle $j: V \ra Y$ defined by 
an integral closed substack $V$ 
of a finite type DM stack $Y$ (see Gillet \cite{Gillet} and Vistoli \cite{Vis} for the definition of the Chow group $A_*(Y)_{\QQ}$ of $Y$ with rational coefficients), 
we let $$\ch ^Y_X (E^{\bullet}) \cap [V] = j'_*  (\ch ^{V}_{V\times _Y X} (j^*E^{\bullet} ) \cap [V]) , $$
where $j': V\times _Y X \ra X$ is the induced inclusion.
Hence, it is enough to define the localized Chern character with assumption that $V=Y$ and $Y$ is irreducible.
When $X=Y$, we define 
\begin{eqnarray}\label{eq:easy case} \ch ^Y_Y (E^{\bullet}) = \ch (E^{+}) - \ch (E^{-}) : A_*(Y)_{\QQ} \to A_*(Y)_{\QQ} \end{eqnarray}
by sending $\gamma$ to $\ch (E^{+}) \cap \gamma - \ch (E^{-}) \cap \gamma$.

When $X\ne Y$, we consider a graph construction for the homomorphism $(d^{+}, d^{-})$ as follows.
Let $r$ be the rank of $E^{+}$. Note that the rank of $E^{-}$ is also $r$. 
Denote by $\GG$ the Grassmann bundle $\mathrm{Gr}_{r}(E^{+}\oplus E^{-})$ of $r$-planes in $E^{+}\oplus E^{-}$.
Consider the projection $$ \pi: \GG \times _Y \GG \times \mathbb{A} ^1 \to Y \times \mathbb{A}^1$$
and an its section 
\begin{eqnarray*} \varphi : Y\times \mathbb{A}^1 & \ra & \GG\times _Y \GG  \times \mathbb{A} ^1 \\
               (y, \lambda) & \mapsto &  (\mathrm{graph}(\lambda d^{+}(y)), \mathrm{graph}(\lambda d^{-}(y))  . \end{eqnarray*}
                       Let $\Gamma$ be the closure of $\varphi (Y\times \mathbb{A}^1)$ in  $\GG \times _Y \GG \times \mathbb{P} ^1$, and
let $$i_{\infty} : \GG \times _Y \GG \times \{\infty\} \hookrightarrow  \GG \times _Y \GG \times \PP ^1$$ be the inclusion.
There is a distinguished component $\Gamma_{\infty, dist}$ of  $\Gamma _{\infty} := \Gamma \ti _{\PP ^1} \{\infty \}$ which birationally projects to $Y$.
The remained components of $\Gamma_{\infty}$ project into $X$. Let $\xi _+$, $\xi _-$ be tautological subbundles on $\GG \times _Y \GG \times \PP ^1$
from each component $\GG$. We consider $\xi := \xi _+  - \xi _-$ as an element in the Grothendieck group of vector bundles  on $\GG \times _Y \GG \times \PP ^1$.
Define 
\begin{eqnarray}\label{eqn:def local ch}  \ch^Y_X(E^{\bullet} )\cap [Y] = \eta _* (\ch (\xi ) \cap ([\Gamma _{\infty}] - [\Gamma_{\infty, dist}  ] ) ), \end{eqnarray}
where $\eta$ is the restriction of the projection $\GG \times _Y \GG \ra Y$ to the inverse image of $X$ under the projection.
Now, for a general morphism $g: Y' \to Y$, define  $\ch^Y_X (E^{\bullet})_g = \ch^{Y'}_{X'} (g^*E^{\bullet})$.

\begin{Rmk}
1. We have a natural projection map $\Gamma_{\infty, dist} \to Y$.
Note that $\Gamma_{\infty, dist}$ restricted to any point $y$ of $Y-X$ is the diagonal point $$ (\Ker\, d^{+}|_{y}\oplus \Ker\, d^{-}|_{y}) \times _Y (\Ker\, d^{+}|_{y}\oplus\Ker\, d^{-}|_{y})$$
 in $\mathbb{G} \ti _Y \mathbb{G}$.
Therefore $\Gamma_{\infty, dist}$ is contained in the diagonal sublocus of $\mathbb{G}\ti _Y \mathbb{G}$.
 Hence $\xi |_{\Gamma _{\infty, dist}} =0$ and 
 the definition \eqref{eqn:def local ch} is justified. 
 
2. Note that $i_* \ch^Y_X (E^{\bullet} ) = \ch (E^{+}) - \ch (E^{-}) $ by Proposition 2.3 (i) of \cite{PV:A}
and \eqref{eq:easy case}. Hence we may regard
$\ch^Y_X (E^{\bullet} )$ as a \lq\lq localized Chern character'' of $E^{\bullet}$.

\end{Rmk}

\subsection{Koszul Complex}\label{Kos}
Koszul complexes yield ample examples of 2-periodic complexes. 
Let $E$ be a vector bundle on $Y$ with sections $\ka \in H^0 (Y, E^\vee)$, $\kb \in H^0(Y,  E )$
such that the pairing $\lan \ka , \beta \ran \in \Gamma( Y, \cO_Y) $ vanishes.
Let $\{\ka , \kb\}$ denote the 2-periodic complex 
\[  \xymatrix{   \oplus _k \bigwedge ^{2k-1} E^{\vee }  \ar@/^.2pc/[r]^{ \alpha\wedge + \iota_\beta  }  
&  \ar@/^.2pc/[l]^{ \alpha \wedge +  \iota_\beta }  \oplus_k \bigwedge ^{2k} E^{\vee }  }
 \]
of vector bundles. Here $\iota _\beta$ is the interior product, i.e., the contraction by $\beta$ defined as 
$\iota_{\beta} (v_1\wedge ... \wedge v_k) = \sum_i (-1)^{i+1} \lan v_i, \beta\ran   \, v_1 \wedge ... \wedge \widehat{v_i} \wedge ... v_k$
for $v_1, ..., v_k \in E^{\vee}$.
This can be regarded as a refined version of the usual Koszul complex given only by $\beta$.

Let $X:=Z (\ka, \kb ):= Z(\ka) \cap Z(\kb )$ be the zero substack of $Y$ defined by
the ideal sheaf generated by the sum of the
images of $E \to \cO$ and  $E^{\vee} \to \cO$ induced by $\ka$ and $\kb$, respectively.

The following lemma shows that the 2-periodic  Koszul complex $\{ \ka, \kb \}$ is strictly exact off $X$. 

\begin{Lemma} \label{lem: strict off}
Let $Z$ be a closed substack of $Y$.
\begin{enumerate}
\item A 2-periodic complex $G^{\bullet}$ of vector bundles on $Y$ is strictly exact off $Z$ if and only if
$G^{\bullet}$ is locally contractible, i.e., there exists an atlas $U$ of $Y-Z$ such that $F^{\bullet}:=G^{\bullet}|_U$ is contractible
by a homotopy map $ h^+\oplus h^- : F^+\oplus F^- \to F^-\oplus F^+$.
\item
The Koszul complex $\{ \ka, \kb \}$  is strictly exact off $X$.
\end{enumerate} 
\end{Lemma}

\begin{proof} 
1) ($\Leftarrow$)
Suppose that there exists an atlas $U$ of $Y-Z$ such that $F^{\bullet}$ is contractible
by a homotopy map $ h^+\oplus h^- : F^+\oplus F^- \to F^-\oplus F^+$. 
Of course, this implies that $G^{\bullet}$ is exact off $Z$.
It remains to verify that  $\Ker d^{\pm}|_{Y-Z}$ are locally free.
Note that the monomorphism $[d^+]: F^+/\Ker d^+ \to F^-$ induced from $d^+_F$ and
the induced map $[h^-]: F^- \to F^+/\Ker d^+$ by $h^{-}$ satisfy together that $$[h^-] \circ [d^+] = \mathrm{id}_{ F^+/\Ker d^+ }.$$
Hence $F^+/\Ker d^+$ is a direct summand of $F^-$, which shows that $F^+/\Ker d^+ \cong \mathrm{Im} d^+ \cong \Ker d^{-}$ is locally free.
Similarly using $h^+$, we can check that $F^-/\Ker d^- \cong \mathrm{Im} d^- \cong \Ker d^{+}$ is also locally free.

($\Rightarrow$) Suppose that $G^{\bullet}$ is strictly exact off $Z$. Then since $\mathrm{Im}d^{\pm} |_{Y-Z}$ are locally free,
the exact sequences $$0 \to \mathrm{Im}d^{\mp} |_{Y-Z} \to G^{\pm}|_{Y-Z} \to \mathrm{Im}d^{\pm} |_{Y-Z} \to 0$$ locally split. 
Hence there exists an atlas $U$ of $Y-Z$ such that $G^{\pm}|_{U} = \mathrm{Im} d^+|_{U} \oplus \mathrm{Im} d^- |_{U}$
and $d^{\pm}|_{U}$ are the projection with kernels $\mathrm{Im}d^{\mp}|_U$. Now it is clear that $G^{\bullet}|_U$ is contractible.

2) This follows from Proposition 2.3.3 of  \cite{MF} and (1). 
\end{proof}

As a special case of Lemma \ref{lem:duality} below,
we note here that there is an isomorphism
\begin{eqnarray} \label{duality} \{ \ka , 0 \} \cong \{ 0 , \ka \} \ot (\bigwedge ^{r} E^{\vee} ) [r]   \end{eqnarray} 
where $r$ is the rank of $E$. This will be used later.

\begin{Lemma}\label{lem:duality}
There is an isomorphism
\begin{eqnarray} \label{full duality} \{ \ka , \beta \} \cong \{ \beta , \ka \} \ot (\bigwedge ^{r} E^{\vee} ) [r]   \end{eqnarray} 
where $r$ is the rank of $E$. 
\end{Lemma}
\begin{proof}
For each non-negative integer $m$, there is  a non-degenerate pairing
\begin{align*} \Lambda ^m E  \ot \Lambda ^m (E^{\vee})  & \ra \cO_{Y}  \  , \\
 [v_1 \ot ... \ot v_m] \ot [v_1^* \ot ... \ot v_m^*] & \mapsto \sum_{\sigma \in \mathfrak{S}_m } \mathrm{sgn}(\sigma ) v_1^*(v_{\sigma (1)}) ... v_m^*(v_{\sigma (m)} )  \end{align*}
and hence  an identification $(\Lambda^m E)^\vee = \Lambda ^m (E^{\vee})$. If $\lan , \ran$ denotes the paring,
the isomorphism \eqref{full duality} is due to the duality of wedge product and interior product:
for $x \in \bigwedge ^l E^{\vee}, v \in \bigwedge^{l+1}  E$,
\begin{eqnarray} \label{eqn:wedge interior}  \lan \alpha \wedge x , v \ran = \lan x, \iota_{\alpha} (v) \ran  \end{eqnarray} 
which follows from the definition of the interior product (see for example Chapter 22 of \cite{tensor alg}). 
We provide the detail of how to get \eqref{full duality} from \eqref{eqn:wedge interior}.

 For each non-negative integer $k$ with $r-k\ge 0$, consider an isomorphism
\begin{eqnarray*}
\varphi_k: \bigwedge ^{k} E \ot  \bigwedge^{r} E^\vee  \ra   \bigwedge ^{r - k}  E^\vee,
\end{eqnarray*}
defined by 
$$\lan \varphi_k( u\ot z ) , v  \ran = \lan z,  v\wedge u \ran \ \forall  u\in  \bigwedge ^{k} E, z \in \bigwedge^{r} E^\vee,
v \in \bigwedge^{r - k} E .$$
Then, by \eqref{eqn:wedge interior}
we have $\varphi_{k-1} ( \iota_\ka (u)  \ot z ) =(-1)^{r-k} \ka \wedge \varphi_{k} ( u \ot z  ) $.  Hence we have 
a commuting diagram
\begin{eqnarray} \label{com diag 1}
\xymatrixcolsep{5pc} \xymatrix{
\bigwedge ^{k} E  \ot  \bigwedge^{r} E^\vee  \ar[r]^-{(-1)^{r} \cdot  \iota_\ka  \ot \mathrm{id}} \ar[d]^-{s(k) \varphi_k} & \bigwedge ^{k-1} E \ot  \bigwedge^{r} E^\vee    
\ar[d]^-{s(k-1) \varphi_{k-1} }  \\
\bigwedge ^{r - k}  E^\vee  \ar[r]_-{\ka \wedge}  &  \bigwedge ^{r - k + 1}  E^\vee  ,
}
\end{eqnarray}
where
\[ s(k) := \left\{ \begin{array}{ll} +1 & \text{ if } k \equiv 0 \text{ mod } 4 \\
                                                  +1 & \text{ if } k \equiv 3 \text{ mod } 4 \\
                                                  -1 & \text{ if } k \equiv  2 \text{ mod } 4 \\
                                                  -1 & \text{ if } k \equiv 1  \text{ mod } 4 .\end{array} \right. \]
By \eqref{eqn:wedge interior}, we have also $\varphi_{k+1}(  (\kb \wedge u) \ot z  ) = (-1)^{r-k-1} \iota_ \kb  (\varphi_k( u \ot z    ))$ and hence
a commuting diagram
\begin{eqnarray} \label{com diag 2}
\xymatrixcolsep{5pc} \xymatrix{
\bigwedge ^{k} E  \ot \bigwedge^{r} E^\vee  \ar[r]^-{(-1)^{r} \cdot  \kb \wedge \ot \mathrm{id}   } \ar[d]^-{ s(k) \varphi_k} & \bigwedge ^{k+1} E  \ot \bigwedge^{r} E^\vee \ar[d]^-{s(k+1) \varphi_{k+1} }  \\
\bigwedge ^{r - k}  E^\vee  \ar[r]_-{\iota_\kb}  &  \bigwedge ^{r - k+1}  E^\vee  .
}
\end{eqnarray}
Therefore, by the commuting diagrams \eqref{com diag 1} and \eqref{com diag 2},
$\oplus_{k=0}^r s(k)\varphi_k : \{ \beta , \alpha \} \ot (\bigwedge ^{r} E^{\vee} ) [r]  \to \{ \alpha, \beta \}$ is an isomorphism of the 2-periodic complexes.
\end{proof}

\subsection{Tautological Koszul complex}\label{taut}
Let $M$ be a DM stack and let $F$ be a vector bundle on $M$. 
Consider $\sigma \in H^0(M, F^\vee)$. It yields a  function on the total space $F$:
\[ w_{\sigma}: F  \ra \mathbb{A}^1. \]
Denote by $p$ the projection $F \ra M$. Then 
there is a tautological section $t_F \in H^0(F , p^*F)$ defined by the diagonal morphism $F  \ra F \times_M F$.
Note that $\lan p^* \sigma , t_F \ran = w_\sigma$, where we regard 
the left hand side as the composition $F \xrightarrow{p^*\sigma \ot t_F} p^*F^{\vee} \ot p^* F \xrightarrow{pairing} \mathbb{A}^1$.
Consider the Koszul complex
$$\{ p^* \sigma, t_F \}$$ on the zero locus $Z(w_{\sigma}) := w_{\sigma}^{-1}(0) $ of $w_{\sigma}$.

Starting from the setup in \S \ref{Kos}, we can build the tautological one by
 letting $M :=Z(\kb)$, $F :=E|_{Z(\beta)}$, $\sigma := \ka |_{Z(\beta)}$. Note that $Z(\sigma) = X$.

  Let $C_{V\cap Z(\beta)} V$ denote the normal cone to $V\cap Z(\beta)$ in an integral closed substack $V$ of $Y$.
  If we let $\mathscr{I}$ be the ideal sheaf of the substack $Z(\beta)$  in $Y$ 
and $\mathscr{E}$ denotes the sheaf associated to the vector bundle $E$, 
we have the surjection 
$$ \Sym ^{\bullet} \mathscr{E}^\vee  \ot \cO_{Z(\beta)}  := \oplus _{k=0}^{\infty} \Sym^k \mathscr{E}^\vee  \ot \cO_{Z(\beta)} \to \oplus_{k=0}^{\infty} \mathscr{I}^k \ot \cO_{Z(\beta)}$$
 induced from $\beta$. Hence $C_{Z(\beta)} Y$ can be regarded as  a closed substack of $F=E|_{Z(\beta)}$.
 This in turn implies that $C_{V\cap Z(\beta)} V$ is  a closed substack of $F|_{V\cap Z(\beta)}$ since $C_{V\cap Z(\beta)} V$ is naturally a closed 
 substack of $C_{Z(\beta)} Y \ti _{Z(\beta )} (V\cap Z(\beta ))$.

 \begin{Lemma}\label{taut:lemma}
 The following statements hold.
\begin{enumerate}

\item  The substack $C_{V\cap Z(\beta)} V$ of $F|_{V\cap Z(\beta )}$ is contained in $Z(w_\sigma )$.

\item 
 \[ \ch ^Y_X (\{ \ka , \kb \}  ) \cap [V] =  \ch ^{Z(w_\sigma)} _{Z(\sigma)} (\{ p^* \sigma, t_{F} \} )\cap   [ C_{V\cap Z(\beta)} V] . \]
 \end{enumerate}
 
\end{Lemma}

 \begin{proof}  
 We may assume that $V=Y$ by the closed immersion $C_{V\cap Z(\beta)} V \subset C_{Z(\beta)} Y \ti _{Z(\beta )} (V\cap Z(\beta ))$ 
 and the compatible property with proper pushforward of the localized Chern characters (for (1) and (2), respectively). 
 Consider 
 the graph $\Gamma _{\beta}$ of a section $\lambda \beta$ of $E \times ( \mathbb{A}^1- \{ 0 \})$ on $Y\times ( \mathbb{A}^1- \{ 0 \})$ and its closure $\overline{\Gamma}_{\beta}$
 in $E \times (\PP ^1- \{ 0\}) $, whose fiber at $\infty \in \PP^1 - \{ 0\} $ is  $C_{Z(\beta )} Y $. If $\tilde{p}$ denotes the projection of $E \times \PP ^1 - \{ 0\}$ to $Y$, then
 the vector bundle $(\tilde{p}^*E) |_{\overline{\Gamma}_{\beta}}$ with the diagonal section $\tilde{t}$ and $\tilde{p}^*\ka$ realizes the deformation of $E$ with $\kb, \ka$ to 
 $(p^*E)|_{ C_{Z(\beta )}Y}$ with $t_F, p^*\sigma$.   
 On $\overline{\Gamma}_\kb$, $\lan \tilde{p}^*\ka , \tilde{t} \ran =0$ because it becomes $\lambda \lan \ka, \kb \ran =0 $ at any point $\lambda \in \mathbb{A}^1- \{ 0\}$.
 In particular, $C_{Z(\beta)} Y$ is a substack of $Z(w_\sigma )$, which proves (1).
 Now, the class $\ch ^{\overline{\Gamma}_\kb}_{X \times (\mathbb{P}^1- \{ 0\}) } (\{  \tilde{p}^*\ka , \tilde{t} \}  ) \cap [\overline{\Gamma}_\kb] $ 
 pulled back to $\ch ^Y_X (\{ \ka , \kb \}  ) \cap [Y]$ at any point $\lambda \in \mathbb{A}^1-\{0\}$, and pulled back to $\ch ^{C_{Z(\beta)} Y} _{Z(\sigma)} (\{ p^* \sigma, t_{F} \} )\cap   [ C_{ Z(\beta)} Y] 
 \stackrel{(\star)}{=} \ch ^{Z(w_\sigma)} _{Z(\sigma)} (\{ p^* \sigma, t_{F} \} )\cap   [ C_{Z(\beta)} Y]$ at $\infty \in \PP^1$.
Here the equality $(\star) $ follows  from the compatibility with proper pushforward. 
Both pullbacks are equal by the compatibility with refined Gysin homomorphism, i.e., Definition 17.1 $(\mathrm{C}_3)$ 
of \cite{Ful}. This proves (2).
\end{proof}

 Let $j: X:=Z(\ka, \kb) \hookrightarrow Z(\beta)$ be the inclusion.
Then the following corollary shows that $\ch ^{Y}_X \{ \ka, \kb\} \cap [Y]$ after pushforward by $j$ is
nothing but the localized top Chern class  of $E$ up to a Todd class operation.

 \begin{Cor}\label{Gysin}
 We have
\begin{equation} \label{push to amb} j_*(\ch^{Y}_X (\{ \ka, \kb \})  \cap [V])
= (\td E|_{Z(\beta)} )^{-1} \cdot  0^!_{E|_{Z(\beta)}} ([C _{V\cap Z(\beta)} V]) . \end{equation}
Furthermore if $\beta$ is regular so that the natural homomorphism from $E^{\vee}|_{Z(\beta )}$ to the conormal sheaf
of $Z(\beta)$ in $Y$ is an isomorphism (see \cite[\S A.5]{Ful} for the definition of a regular section), 
then
\begin{equation} \label{push to amb regular} j_*\ch^{Y}_X (\{ \ka, \kb \}) 
= (\td  E|_{Z(\beta)} )^{-1}  \cdot  i_{Z(\beta)}^! \end{equation}  where 
$i_{Z(\beta)}$ denotes the regular immersion of $Z(\beta)$ in $Y$.
\end{Cor} 
 \begin{proof} 
  To prove \eqref{push to amb}, we may assume $V=Y$ using the bivariant properties of both side.
 By applying Lemma \ref{taut:lemma}, Proposition 2.3 (i) of \cite{PV:A}, and the homotopy deforming $p^*\sigma$ to $0$, we obtain
 \begin{align*}
 j_*(\ch^{Y}_X (\{ \ka, \kb \})  \cap [Y]) & = j_*( \ch ^{Z(w_\sigma)} _{Z(\sigma)} (\{ p^* \sigma, t_{F} \} )\cap   [ C_{ Z(\beta)} Y] ) \\
 & =  \ch ^{Z(w_\sigma)} _{Z(\kb)} (\{ p^* \sigma, t_{F} \} )\cap   [ C_{ Z(\beta)} Y] \\
 & =  \ch ^{Z(w_\sigma)} _{Z(\kb)} (\{ 0, t_{F} \} )\cap   [ C_{ Z(\beta)} Y].
 \end{align*}
 Note that $\{ 0, t_{F} \}$ is the 2-periodic complex corresponding to the Koszul-Thom complex.
 Now Proposition 2.2,  Proposition 2.3 (vi) of \cite{PV:A}, and the compatibility with proper pushforward  complete the proof of \eqref{push to amb}. 
 The equation \eqref{push to amb regular} is immediate from \eqref{push to amb} since $$0^!_{E|_{Z(\beta)}} ([C_{V\cap Z(\beta)} V]) = i^!_{Z(\beta )} ([V])$$ for the regular $\beta$.
  \end{proof}
 
Let $V$ be an integral substack of $Y$. 
For the pair $(E|_{V}, \beta |_V)$, there is the notion of the localized top Chern class  of $E|_V$ with respect to $\beta|_V$; see \cite[\S 14.1]{Ful}.
It is by definition $0^!_{E|_{Z(\beta)}} ([C _{V\cap Z(\beta)} V]) \in A_*(Z(\beta ))_{\QQ}$.
This eventually yields a bivariant class in $A^*(Z(\beta ) \to Y)_{\QQ}$, which we call 
the {\em localized top Chern class operation} of $E$ with respect to $\beta$. 

\begin{Cor} 
The class $\td (E|_{Z(\beta)}) \cdot \ch^{Y}_{Z(\beta)} (\{ 0, \kb \})$ agrees with the localized top Chern class operation of $E$ with respect to $\beta$.
\end{Cor}
 
 \begin{proof}
 This is immediate from \eqref{push to amb} for $\ka =0$ and the definition of the localized top Chern class. 
 \end{proof}

 \subsection{Splitting Principle}
 
 The splitting principle shows that essentially localized Chern character operation for a 2-periodic Koszul complex is a composition of 
 localized top Chern class operations, one given by a section and the other given by a cosection, up to Todd correction.

 \subsubsection{} \label{deformation: spl}
 Consider the situation of \S \ref{Kos}.
 From now, for a section $\alpha \in H^0(Y, E^\vee)$, we write the associated cosection schematically as
$\alpha: E \ra \cO_Y$.
{\em Suppose that the cosection $\ka$ is factored through as a cosection $\ka _Q$ of 
a quotient vector bundle $Q$ of $E$.} Let $f: K  \hookrightarrow E$ be the kernel of 
the quotient map $q: E\ra Q$. {\em Furthermore suppose that $q\circ \beta = 0$.}
This means that $\beta$ induces a section $\beta _K$ of $K$.

We consider the vector bundles on $Y\times \mathbb{A}^1$  by the pullback of $E$, $K$ under the projection map. They will be 
denoted by same symbols $E$, $K$ abusing notation. If $\mu$ denotes the standard coordinate of $\mathbb{A}^1$,
we have a commuting diagram of homomorphisms of vector bundles on $Y\times \mathbb{A}^1$:
\[ \xymatrix{                     &                     & & \cO_{Y\times \mathbb{A}^1} \ar[d]_{(\mu\kb, (1-\mu ) \beta _K)}     \ar@{-->}[rrd]^{\kb_P} &  &        &  \\
       0 \ar[r] &    K \ar[rr]_{(f, \mu\mathrm{id}_K)}  & & E \oplus K \ar[rr]  \ar[d]^{(\ka , 0)}                             & &  P     \ar@{-->}[lld]^{\ka _P}    \ar[r]  &  0   \\
                                &                             & &                                        \cO_{Y \times \mathbb{A}^1}                                    &    &             &        }   \] 
 and the induced section $\beta _P$ and cosection $\alpha _P$ of $P$. Here $P$ is defined to be 
 the cokernel of  $(f, \mu\mathrm{id}_K)$. 
 
 Note that $P$ restricted to $\mu =0$ is canonically isomorphic to
 $ Q \oplus K $; and $P$ restricted to any nonzero $\mu$ is canonically isomorphic to $E$.
 Note that $Z(\ka _P, \kb _P )$ coincides with  $X \times \mathbb{A}^1$ set-theoretically. 
 Hence $\{ \ka _P, \kb _P\}$ is strictly exact off $X \times \mathbb{A}^1$. Using the bivariance of $\ch^{Y\ti \mathbb{A}^1}_{X\ti \mathbb{A}^1}\{ \ka _P, \kb _P\}$ 
 with refined Gysin maps,
 we have that
 
\begin{equation}\label{eqn: spl dec} \ch ^Y_X \{ \ka , \kb\} =  \ch ^Y_X (\{\ka _Q, 0\} \ot \{ 0 , \beta _K\} ) . \end{equation}

 \medskip
 
 The Chern characters of $\{\ka _Q, 0\}$ and $\{ 0 , \beta _K\}$ can be expressed as refined Gysin maps as in Lemma \ref{SP1} below.
 For the precise statement we first introduce some notation.

 Let $\cE$ be a vector bundle on a DM stack $\bfB$ and let $\bfA$ be the zero locus of a section of $\cE$.
 We denote by $\spp _{C_{\bfA}\bfB}$  the specialization homomorphism $A_* (\bfB)_{\QQ}  \ra A_* (C_{\bfA}\bfB)_{\QQ}$
 followed by the pushforward to $A_*(\cE|_{\bfA})_{\QQ}$ under the inclusion $C_{\bfA}\bfB \subset \cE|_{\bfA}$.

 \begin{Lemma}\label{SP1} The following equality holds:
 \begin{align*}    
  \ch ^Y_{X} (\{ \ka, \kb \} ) \cap [V]  = 
 \td (E|_{X})^{-1} (-1)^{\rank Q} 
 0^{!}_{Q^{\vee}|_{Z(\beta _K, \ka _Q)}} \\
 ( \spp _{ C_{Z(\beta _K, \ka _Q)}  Z(\beta _K) } (  0^{!}_{K|_{Z(\kb _K)}} 
 [ C_{V\cap Z(\kb _K ) } V ] ) ) . \end{align*}
\end{Lemma}
 
 \begin{proof} Let $p$ be the projection $|K| \ra Y$. Note that
  \begin{align*} & \ch ^Y_{X} (\{ \ka, \kb \} ) [V]  
                 \\  = &  \ch ^Y_{X} ( \{\ka _Q , 0 \}  \ot \{ 0, \kb _K\}) [V]  \nonumber \\
                                   = & \ch ^{ K |_{Z(\beta_K)}  }_{X}  ( \{ p^*\ka _Q , 0\} \ot \{ 0, t_K \}) [ C_{V\cap Z(\kb _K ) } V ]  \nonumber \\    
                                   = &  (\td K|_{X})^{-1} \ch ^{Z(\beta _K)}_{X}  ( \{ \ka _Q , 0\} ) \cdot  0^{!}_{K|_{Z(\kb _K)}}[ C_{V\cap Z(\kb _K ) } V ] .
                                  \end{align*}
The first equality is from \eqref{eqn: spl dec}. 
The second equality is the deformation to the normal cone.
The last equality is 
Proposition 2.3 (vi) of \cite{PV:A}.
Finally using \eqref{duality}, Corollary \ref{Gysin}, and the fact that $\td \cE = \td \cE^\vee \cdot (\ch (\Lambda ^{\rank \cE} \cE^\vee ))^{-1}$ for a vector bundle $\cE$, we conclude the proof.
\end{proof}

\begin{Cor}\label{Case:regular}
If $\ka$, $\kb$ are regular sections,
then  \[ \ch ^Y_{X} (\{ \ka, \kb \} )   =(-1)^{\rank Q} \td (E|_{X})^{-1}  i_{Z(\ka _Q, \kb _K)}^!   i_{Z(\kb _K)}^!  , \]
where $i_{Z(\ka _Q, \kb_K)}, i_{Z(\kb _K)}$ are regular immersions of $Z(\ka _Q, \kb _K), Z(\kb _K)$ into $Z(\kb_K)$, $Y$ respectively.
\end{Cor}
\begin{proof}
This immediately follows from Lemma \ref{SP1} and the definition of refined Gysin homomorphisms; see \S 6.2 of Fulton \cite{Ful}.
\end{proof}
 \subsubsection{} Consider the situation of \S \ref{taut}.
 
 Let $W$ be an integral closed substack of the total space of $F$ which is factored through $p^{-1}(Z(\sigma))$,
  i.e., $W \ra p^{-1}(Z(\sigma )) \subset  F$. Note that this means that
  $p^*\sigma |_W = 0$. Hence by Proposition 2.3 (vi) of \cite{PV:A}
  \begin{eqnarray} \label{Trivial_alpha} \ch ^{Z(w_{\sigma})}_{Z(\sigma)} (\{ p^*\sigma, t_F\}) \cap [W] = \td(F|_{Z(\sigma)} )^{-1}  \cdot 0^!_{F|_{Z(\sigma )} }([W]) . \end{eqnarray}

 Let $W$ be an irreducible cycle of $F$ which is not factored through $p^{-1}(Z(\sigma))$. 
 Following \cite{KiemLi, CLL}, we consider the blow-up $M'$ of $M$ along $Z(\sigma)$.
  Let $F'$ be the pullback of $F$ to $M'$, and let $D$ be the exceptional divisor.
  On $M'$ we obtain short exact sequences and a chain map between them 
   \[ \xymatrix{   0 \ar[r] &        K    \ar[r] \ar[d] &      F'   \ar[r] \ar[d]_{\sigma '}  & \cO_{M'} (-D)    \ar[d]_{s_D} \ar[r] & 0 \\
                             0 \ar[r] &    0      \ar[r]          &    \cO_{M'}  \ar@{=}[r] & \cO_{M'}  \ar[r] & 0 ,   } \]
                             where $K$ is defined to be the kernel, $\sigma '$ is the pullback of $\sigma$, and $s_D$ is the inclusion map of the ideal sheaf $\cO_{M'}(-D)$ of $\cO_{M'}$.
This shows that locally $(F', w_{\sigma '})$ is isomorphic to $(K \oplus \cO_{M'} (-D),  w_{s_D})$. 
Therefore we note that the proper transform $W'$ of $W$ is contained in $K$ since it is the case for general points of $W'$. 
 By the compatibility with proper pushforward $Z(w_{\sigma '}) \ra Z(w_\sigma )$, we have
 $$\ch ^{Z(w_{\sigma})}_{Z(\sigma)} (\{p^*\sigma, t_F\} ) \cap [W]  = b_* \ch _{Z(\sigma ' )} ^{ K   } (\{ (p')^*\sigma ', t_{F'}\} )  \cap [W'] $$ 
 where $b: Z(\sigma ' ) = D \ra Z(\sigma )$ and $p' : F' \ra M'$ are the projections.

By Corollary \ref{Case:regular} we conclude that
\begin{align}\label{SP3}   \ch ^{Z(w_{\sigma})}_{Z(\sigma)}  (\{p^*\sigma, t_F\} ) \cap [W]  & 
	    =  - (\td E|_{Z } )^{-1} \cdot b_* (i_{D}^! \cdot i_{M'}^! ([W'])) 
\end{align}
where $i_D: D \ra M' $ is the inclusion and  $i_{M'} : M' \ra K$ is the inclusion as the zero section.

\subsection{Cosection Localization}

Consider the setup in \S \ref{taut}.
Kiem and Li \cite{KiemLi} defined the cosection localized Gysin map:
$0^!_{F, \sigma } : A_*(Z(w_\sigma ))_{\QQ} \ra A_*(Z(\sigma ))_{\QQ}$, for an algebro-geometric understanding of
a work of Lee and Parker \cite{LeeParker}.

\bigskip

\noindent{\em Proof of Theorem \ref{Thm1}.}
The equations  \eqref{Trivial_alpha} and  \eqref{SP3}  exactly match with basic construction 
determining the cosection localized Gysin map; see \S 2 of \cite{KiemLi}. \hfill $\Box$

\bigskip

Let $d: A\ra F$ be a complex of vector bundles on a DM stack $M$. Suppose that its dual gives rise to a 
perfect obstruction theory relative to a pure-dimensional stack $\fM$. 
Supposed that $M\ra \fM$ is representable.
Consider a cosection $\sigma$ of $F$ such that $\sigma \circ d  = 0$.
Let $C$ be the cone in $F$ associated to the relative intrinsic normal cone of $M$ over $\fM$. 
Assume that the cosection has a lift as a cosection of {\em absolute} obstruction sheaf. 
Then
$C$ is as a cycle, i.e., set-theoretically, supported in $Z(w_{\sigma} )$
by Kiem - Li  \cite[Proposition 4.3]{KiemLi}.

As the immediate consequence of Theorem
\ref{Thm1} we obtain the following corollary.

\begin{Cor}\label{CoVir:Ch} The following equality holds:
\[ [M]^{\vir}_{\sigma} := 0^!_{F, \sigma } [C]  = \td  F|_{Z(\sigma)}  \cdot \ch^{Z(w_\sigma )}_{Z(\sigma )} ( \{ p^* \sigma , t_F \})\cap [C] . \]
\end{Cor}

By Corollary \ref{CoVir:Ch} and Lemma \ref{taut:lemma} (2) we obtain this.
\begin{Cor}\label{CLL:Prop} {\em (Chang, Li, and Li \cite[Proposition 5.10]{CLL})}
Consider the set-up in \S \ref{Kos}. Suppose that $Y$ is smooth.
Let $F=E|_{Z(\beta)}$ and $\sigma = \ka |_{Z(\beta)}$.
Then $$ 0^!_{F, \sigma } [C_{Z(\beta )} Y] = \td E|_{Z(\ka, \kb)} \cdot  \ch ^Y_{Z(\ka, \kb)} (\{ \ka, \kb \}) \cap  [Y] . $$ 
\end{Cor}

\begin{Rmk}
The difference between Theorem \ref{Thm1} and Corollary \ref{CLL:Prop} is that the latter assumes that $Y$ is smooth. 
In section \S \ref{section:com_vir}, we will need Theorem \ref{Thm1}.
\end{Rmk}

\section{Comparisons of virtual classes}\label{section:com_vir}

We apply the bivariant property of localized Chern characters to the comparison of certain virtual classes.
In this section, let the base field $\bf k$ be the field of complex numbers.
\subsection{Conjecture}

Let $V_1, V_2$ be vector spaces over $\bf{k}$ and let a reductive algebraic group $G$ 
act on $V_1$ and $V_2$ linearly. Fix  a character $\theta$ of $G$ such that $V_1^{ss}(\theta) = V_1^{s}(\theta)$, i.e., there is no strictly semistable
points of $V_1$ with respect to $\theta$.
 Let $E:=[(V_1^{ss}(\theta)  \times V_2) / G ] $, $\bar{E} := [V_1\times V _2 /G]$, $V_1\ti V_2$
which are vector bundles on stack quotients $[V_1^{ss}(\theta) /G]$, $[V_1/G]$, $V_1$ respectively. 
Fix $w  \in ((\Sym^{\bullet \ge 1} V^{\vee}_1) \ot V_2^\vee)^{G}$. 
The polynomial $w$ induces sections $s$, $\bar{s}$ of $E^\vee$, $\bar{E}^\vee$  and also  morphisms $\uf$, $\underline{w}$ below:
       \[ \xymatrix{   \bar{E}^\vee  \ar[d]  & \ar@{_{(}->}[l] E^\vee  \ar[d]  \\ 
                         [V_1/G]    \ar@/_1pc/[u]_{\bar{s}}     & \ar@{_{(}->}[l]  [V_1^{ss}(\theta)/G] \ , \ar@/_1pc/[u]_{s}        } 
         \ \ \      \  \ \ \ \       \xymatrix{  V_1 \ar[r]^{\uf} &  V_2^\vee \ , \\
                                       E \ar[r]_{\underline{w}} & \mathbb{A}^1 \ , }  \]
                                       where $\uf$ is defined by the element corresponding to $w$
                                       in $$\Hom _{Alg} (\Sym ^{\bullet} V_2, \Sym ^{\bullet} V_1^{\vee}) = \Hom _{\CC} (V_2, \Sym ^{\bullet} V_1^{\vee} ) .$$ 
  This is so-called a geometric phase of a hybrid gauged linear sigma model; see \cite{MF}.
  Let $r$ be the dimension of $V_2$, i.e., the rank of $E$.
  We require that the critical locus $Z(d\underline{w})$ of the function $\underline{w}$ is a smooth closed locus in the zero section locus $[V_1^{ss}(\theta)/G] \subset  E^\vee $
  with codimension $r$. 
Note that canonically $Z(d\underline{w}) = Z(s)$.  

\subsubsection{Tangent Complex of $Z(\uf )$}
Let $\ke \in  \QQ _{> 0}$. 
We consider the moduli space $Q^{\ke}_{g, k} ( Z(s), d)$ of $\ke$-stable quasimaps to $Z(s)$
with type $(g, k, d)$ where $g$ is genus, $k$ is the number of markings, $d \in \Hom ( \hat{G} , \QQ   )$ is a fixed curve class  (see \cite{MOP, CKM}).
Here $\hat{G}$ is the character group of $G$.
The stable quasimaps to $Z(s)$ are certain maps to the Artin stack $Z(\bar{s})$, not necessarily to $Z(s)$.
The moduli space is a separated DM stack 
 over the affine quotient $\Spec  (\Sym^{\bullet} V_1^\vee )^G$. 
It comes with a canonical virtual fundamental class denoted by $[Q^{\ke}_{g, k} ( Z(s) , d)]^{\vir}$; see \cite{CKM, CCK}.

\bigskip

\noindent{\em Conventions}: 
Let $\fM _{g, k}(BG, d)$ be the moduli space of principal $G$-bundles $P$ on genus $g$, $k$-marked prestable orbi-curves $C$ with degree $d$ such that
the associated classifying map $C \ra BG$ is representable.
The algebraic ${\bf k}$-stack $\fM _{g, k}(BG, d)$ is smooth; see \cite{CKM, CCK}.
Let $\cP$ be the universal $G$-bundle on $\fC$ and let ${\bf u} : \fC \ra \cP \times_G V_1$ be the universal section.  
Let $\pi$ be the universal curve map  $\fM _{g, k}(BG, d)$. Let $\cV_1 := \cP \times _G V_1$ and $\cV_2 : = \cP \times _G V_2$.
 By abusing notation, $\pi$ will also denote the universal curve on various moduli spaces over $\fM _{g, k}(BG, d)$. For example,
$\pi: \fC \ra Q^{\ke}_{g, k} ( Z(s), d)$ denotes also the universal curve and $\cP$ denotes also the universal bundle on this $\fC$.
Therefore, we may consider also $\cV_i$ as a vector bundle on the universal curve over $\fC \ra Q^{\ke}_{g, k} ( Z(s), d)$.

\bigskip

Consider a complex of cotangent bundles
\begin{align}\label{ds:complex} V_1\times V_2 = \uf^*\Omega_{V_2^\vee}  \xrightarrow{d\uf}  \Omega_{V_1} = V_1 \times V_1^\vee .  \end{align}
Its dual restricted to the affine scheme $Z(\uf ) \subset V_1$ is the tangent complex of $Z(\uf )$ since it is a complete intersection scheme.
By pulling back the dual of  \eqref{ds:complex} to the universal curve over  $Q^{\ke} _{g, k}( Z(s) , d)$ and then pushforward by $\pi$
we obtain \begin{align} \label{Obst} \RR \pi _* ({\bf u}^* (\cP \ti _G d\uf ^\vee)) : 
 \RR \pi _* \cV_1  \ra  \RR \pi _*  \cV_2^\vee  ; \end{align}
see the proof of Proposition 4.4.1 of \cite{CKM}.
The dual of \eqref{Obst} is the canonical perfect obstruction theory for $Q^{\ke}_{g, k} ( \Zs , d)$ {\em relative to $\fM _{g, k}(BG, d)$}, defining 
$[Q^{\ke}_{g, k} ( \Zs , d)]^{\vir}$.

\subsubsection{LG quasimaps}\label{LG:Moduli}
On the other hand, we may consider the moduli space $LGQ^{\ke}_{g, k} (E, d)'$, for short $\LGQ '$, 
of genus $g$, $k$-pointed, degree $d$, $\ke$-stable quasimaps to $V_1/\!\!/G$ with $p$-fields; see \cite{CL, FJR:GLSM, MF}. 
Here by a $p$-field we mean an element in $H^0(C, \cV_2|_{C} \ot \omega _C)$, where $C$ is a
domain curve.  Note here that we use $\omega$ instead of $\omega _{\log}$ which is used usually in gauged linear sigma model (\cite{FJR:GLSM}). 
For simplicity, let us call $\LGQ '$ the moduli of LG quasimaps to $E$.
Due to the twisting by $\omega _C$, an LG quasimap to $E$ is not a map to $E$ even for larger enough $\ke$.

Let  $\mathbb{L} ^{\bullet}_{LGQ' / \fM _{g, k}(BG, d)}$  denote the cotangent complex of $LGQ'$ relative to $\fM _{g, k}(BG, d)$.
By the same idea of \cite{CCK}, it is clear that $LGQ'$  comes with a perfect obstruction theory 
\begin{align} \label{LGObst}  \RR ^{\bullet} \pi_* (\cV_1\oplus \cV_2\ot \omega _{\fC}) ^\vee  \ra \mathbb{L} ^{\bullet}_{LGQ' / \fM _{g, k}(BG, d)}  
\end{align}
relative to $\fM _{g, k}(BG, d)$.
By \cite{CL, FJR:GLSM, MF} there is a cosection
$$dw_{\LGQ '}: \RR ^1\pi_* (\cV_1\oplus \cV_2\ot \omega _{\fC}) \ra \cO_{LGQ'} , $$
where, by abusing notation, $\cV_i := \cP \times _G V_i$ with $\cP$ the universal $G$-bundle on the universal 
curve $\fC$ on $\LGQ '$. 

We recall the definition of $dw_{\LGQ '}$. Fix a positive integer $d_f$ such that $w \in \oplus _{a=1}^{ d_f}  (\Sym ^{a} V_1^{\vee}) \ot V_3$.
Let  $\Sym^{\star} V_1:= \oplus _{a=0}^{d_f} \Sym^a V_1$ and from what follows we will use the
duality $\Sym ^m  V_1 ^{\vee}  = (\Sym ^m V_1 )^{\vee}$ 
for each non-negative integer $m$ by the non-degenerate pairing
\begin{align*} \Sym^m V_1 \ot \Sym^m (V_1^{\vee})  & \ra \CC  \  , \\
 [v_1 \ot ... \ot v_m] \ot [v_1^* \ot ... \ot v_m^*] & \mapsto \sum_{\sigma \in \mathfrak{S}_m }  v_1^*(v_{\sigma (1)}) ... v_m^*(v_{\sigma (m)} )  . \end{align*}

From the differential of $w$ we may consider ${\bf k}$-linear map
$$ dw : \Sym^{\star-1} V_1 \ot V_1 = (\Sym^{\star-1} V^{\vee}_1 \ot V^{\vee}_1 )^{\vee} \to V_2^\vee, $$
 which induces a map in derived category
 \begin{equation} \label{Rpidw} \RR \pi _* dw : \Sym^{\star-1} \RR \pi_* \cV_1 \ot \RR \pi_* \cV_1 \to \RR \pi_* \cV_2^\vee . \end{equation}
 We refer Section 4.1 of \cite{PV:A} for the definition of $\Sym^{\star -1} \RR \pi_* \cV_1 := \oplus_{i=0}^{d_f-1} \Sym^{i} \RR \pi_* \cV_1 $
 as a complex well-defined up to quasi-isomorphisms. Later we will use also 
 the complex $\Sym^{\star} \RR \pi_* \cV_1 := \oplus_{i=0}^{d_f} \Sym^{i} \RR \pi_* \cV_1 $.

 We consider  the truncated exponential maps $\exp^{\star}, \exp^{\star-1}$ defined by 
\begin{equation} \label{expmap}
\exp ^{\star}, \exp^{\star-1}: \RR ^0  \pi_* \cV_1 \ra \Sym^{\star} \RR ^0 \pi_* \cV_1, \ u \mapsto \sum_{i=0}^{d_f} \frac{u^{\ot i}  }{i!} , \sum_{i=0}^{d_f-1} \frac{u^{\ot i}  }{i!}
\end{equation} respectively
 and  the  Serre-Grothendieck-Verdier duality pairing $\Res$
\begin{equation} \label{ResPairing} \RR ^1 \pi_* \cV _2 ^\vee   \ot \RR ^0 \pi _* (\cV_2 \ot \omega _{\fC})  \oplus \RR ^1 \pi_* (\cV _2  \ot \omega _{\fC}) \ot \RR ^0 \pi _* \cV_2 ^\vee 
\to \cO . 
\end{equation}
Let $$ H^1(\RR \pi _* dw ) : \RR^1 \pi_* \cV_1 \ot \Sym^{\star - 1} \RR^0\pi_* \cV_1 \to \RR ^1\pi_* \cV_2^{\vee}$$ be the part of the $H^1$ of the map \eqref{Rpidw}.
Then for $(u, p) \in LGQ'$ omitting a curve, markings and a principal bundle for easy notation, and for $(u',p')\in \RR ^1\pi_* (\cV_1 \oplus \cV_2\ot \omega _C )$, we define
\begin{equation}\label{def:dw} dw_{\LGQ '}|_{(u, p)} (u', p') := \Res ( \left( H^1(\RR \pi _* dw )(u' \ot \exp^{\star - 1} (u)) \right) \ot p  +  p' \ot \uf (u) ) .  \end{equation}

\begin{example} For $G=\GG _m$ the multiplicative group, 
$dw_{\LGQ '}$ has the following explicit description:  for $(u', p') = (u'_i, p'_j)_{i, j} \in \RR ^1\pi_* (\cV_1\oplus \cV_2\ot \omega _{\fC})$ 
at $(u, p) = (u_i, p_j)_{i, j} \in \RR ^0 \pi_* (\cV_1\oplus \cV_2\ot \omega _{\fC})$,
\begin{align*}  \dwq |_{(u, p)} (u', p')   = \Res (   \sum_{i, j}  u'_i  \frac{\partial \uf _j (u)}{\partial u_i}  \ot p_j + \sum_j  p'_j \ot  \uf _j (u)   )  \end{align*}
where $i, j$ run for $1, ..., \dim V_1$, $1, ..., \dim V_2$, respectively, and $\uf=(\uf_j)_j$. 
\end{example}

Let a global vector bundle complex $[\cF^0 \ra \cF^1]$ represent $\RR \pi _* (\cV_1\oplus \cV_2\ot \omega _{\fC})$.
Then induced from $\dwq$ there is a cosection   of $\cF^1$, which will be denoted also by $\dwq$:
\begin{equation} \label{dwq: 2} \dwq : \cF^1 \to \cO_{LGQ'} . \end{equation}
The zero locus of the cosection $\dwq$, i.e., the locus defined by the ideal sheaf $\mathrm{Im} (\dwq )$, coincides with 
$Q^{\ke}_{g, k} ( \Zs , d)$.  This is a special case of Proposition 3.6.1 of \cite{MF}. Thus we have
 \[ Z(\dwq) =  \{ (u, p ) \in \LGQ' :  \uf (u) = 0 , p = 0 \} = Q^{\ke}_{g, k} ( \Zs , d). \]

Let $p: \cF^1 \ra \LGQ '$ be the projection.  It can be shown that $p^* \dwq \circ t_{\cF^1} = 0$ on the {\em support} of the {\em relative} intrinsic normal cone of $\LGQ '$ over $\fM_{g, k}(BG, d)$;
see \cite{FJR:GLSM}. In fact it will be shown later in Corollary \ref{van} that $p^* \dwq \circ t_{\cF^1} = 0$ on the cone; see also \eqref{Res:rest} to check why 
the cosection descends to that of the absolute obstruction sheaf. 
Hence by the cosection localization method or equivalently by  applying the localized Chern character of  $\{p^* \dwq, t_{\cF^1}\}$ to the obstruction cone in $\cF^1$ (see Corollary \ref{CoVir:Ch}),
we obtain a virtual class  $[\LGQ ']^{\vir}_{\dwq }$ supported in $Z(p^* \dwq, t_{\cF^1}) = Z(\dwq)$.

According to Chang and Li \cite{CL}; and Fan, Jarvis, and Ruan \cite{FJR:GLSM}, we expect the following. 
\begin{Conj}\label{conj}
In $A_*(Q^{\ke}_{g, k} ( \Zs , d))_{\QQ}$,
\begin{align}\label{Vir:Eq} [Q^{\ke}_{g, k} ( \Zs , d)]^{\vir} = (-1)^{\chi ( \cV ^{\vee}_2 ) } [LGQ^{\ke}_{g, k} (E, d)']^{\vir}_{dw_{LGQ'}} 
\end{align}
where $\chi ( \cV ^{\vee}_2 )$ is the virtual rank of  $\RR \pi _* \cV ^{\vee}_2$.
\end{Conj}

For a smooth quintic $\Zs$ in $\PP^4$, a pioneering work of Chang and Li \cite[Theorem 1.1]{CL} shows that
Conjecture \ref{conj} with $k=0$, $\ke >> 0$, and $d>0$ holds true numerically, i.e., after passing to the singular homology:
$A_0(Q^{\infty}_{g, 0} ( \Zs , d))_{\QQ} \to H_0 (Q^{\infty}_{g, 0} ( \Zs , d), \QQ ) = \QQ$ for $d> 0$. Here $\infty$ means $\ke$ is 
large enough.

\subsection{Proof of the Conjecture}

Before proceeding our proof, we remark the other's works on the conjecture. 

\bigskip

\begin{Rmk} After an announcement of the above result,
F. Janda told the first author that he, Q. Chen, and R. Webb are working 
on a proof of the conjecture using torus localization for cosection 
localized virtual classes \cite{ChangKiemLi}.\footnote{Their paper is  appeared in \cite{CJW}.}
After our paper is appeared in arXiv, the paper of H. Chang and M. Li  \cite{ChangML} appeared also in arXiv, 
showing the above result when $[V_1^{ss}/G]$ is the projective space
with $\G=\mathbb{G}_m$ and $Z(s)$ is a hypersurface.  
Their proof uses, among other things from the original proof of H. Chang and J. Li \cite{CL} in a special case, 
the degeneration of \lq target' 
$E$ to the normal cone $C_{Z(s)}E$ . 
A similar degeneration appears also in our proof, too; see $\tilde{U}_{PQ}$ and $\beta$ in
\S \ref{M_degen}. There is another paper \cite{Picc} by R. Picciotto on the comparison showing a generalization of  quantum Lefschetz hyperplane principle.  
\end{Rmk}

\bigskip

For easy notation, let  $\fB := \fM _{g, k}(BG, d)$, $X:=Z(s)$, $Q_X^{\ke} := Q^{\ke}_{g, k} ( X, d)$,
$Q_{V_1}^{\ke} := Q^{\ke}_{g, k} ( V_1/\!\!/G , d)$;
 and let $\pi: \fC \ra \fB$ be the universal curve.

\subsubsection{Construction of $\phi _{A_1}$, $\bar{\phi} _{B_1}$} \label{CochainReal}
For simple notation, let $V_3$ be the dual vector space of $V_2$.
Recall that $\Sym ^{\star}$ denotes $\oplus _{a=0}^{d_f} \Sym ^{a}$ for some positive integer $d_f$ 
such that $w \in \oplus _{a=0}^{ d_f}  (\Sym ^{a} V_1^{\vee}) \ot V_3$.
Let $f: \Sym^{\star} V_1 \to V_3$ be a linear map induced from $w\in \Sym^{\star} V_1^\vee \ot V_3$.
Combining with the natural homomorphism $nat$ (see Section 3.2.1 of \cite{MF} for the definition) we get
\begin{align}\label{der:maps}  \Sym^{\star} \RR \pi_* \cV_1 \xrightarrow{nat} \RR \pi_* \Sym ^{\star} \cV_1 \xrightarrow{\RR \pi _* f}  \RR \pi _* \cV_3  \end{align}
on $\fB$ where $\cV_3 := \cV_2 ^\vee$. Here for the definition of 
$\Sym^{\star} := \oplus _{a=0}^{d_f}\Sym ^a$ operator (in particular for two-term complexes), see \S 4.1 of \cite{PV:A}.
The maps in \eqref{der:maps} are maps in the derived category of coherent sheaves. 
We seek for the cochain maps which represent those maps at least locally by the following two steps.

\medskip

{\em Step 1.} For some positive integer $m_0$, 
let $\cO (1) := (\omega _{\fC}^{\log} \ot (\cP \times _G \CC_{\theta} )^{\ke} )^{m_0}$  whose pullback to $LGQ'$ is $\pi$-ample.
We take an open substack of $\fBo$ of $\fB$ such that
the map $LGQ' \ra \fB$ is factored through $\fBo$ and where 
$ \cO (1)$ 
is still $\pi$-ample. We carry out the following construction over $\fBo$.

We first take a $\pi$-acyclic, locally free resolutions of $\cV_1$ for large enough $l$ 
\begin{align*} 0 & \ra  \cV_1 \xrightarrow{h} \cA_1= 
\pi^* (\pi_*(\cO (l)))^\vee \ot \cV_1 (l)  
\ra \cB_1  \ra 0 ,
 \end{align*} 
 where $\cB_1$ is defined to be the cokenel. This is an exact sequence of vector bundles on the universal curve on $\fB ^{\circ}$.
There are the induced homomorphisms $\Sym ^{\star} h: \Sym^{\star} \cV_1\ra \Sym^{\star} \cA_1$ and $ f_{\cV_1}:= \cP \times _G f : \Sym^{\star} \cV_1 \ra \cV_3$.

We next want to construct a $\pi$-acyclic resolution 
of $\cV_3$:
$$0\to \cV_3 \ra \cA_3 \ra \cB_3 \ra 0$$
with a homomorphism $\Sym^{\star} \cA_1 \to \cA_3$ 
compatible with $\Sym ^{\star} h, f_{\cV_1}$ (see \eqref{res:V_3}).
For this construction, we consider the cokernel $\cV_3'$ of $$   (f_{\cV_1}, - \Sym ^{\star} h) :  \Sym^{\star} \cV_1 \to   \cV_3 \oplus   \Sym^{\star} \cA_1   .   $$  
  Note that the induced map $\cV_3 \ra \cV_3'$ is an inclusion of vector bundles.
The locally free sheaf  $\cV_3'$ of finite rank
 has $\pi$-acyclic locally free resolution $0\ra \cV_3' \ra \cA '_3 \ra \cB'_3\ra 0$ by the same method above for a $\pi$-acyclic resolution of $\cV_1$.
Let us take $\cA_3:= \cA '_3$. Consider the map $\cV_3 \ra \cA_3$ which is given by the composition of 
 inclusions $\cV_3 \ra \cV_3' \ra \cA_3$. 
 This gives rise to a $\pi$-acyclic, locally free resolutions of $\cV_3$
 \begin{align*} 0 & \ra  \cV_3 \ra \cA_3  
 \ra \cB_3  \ra 0 ,
 \end{align*}
 where $\cB_3$ is defined to be the cokernel.  
 Note that $\cB_3$ is $\pi$-acyclic since $\cA_3$ is $\pi$-acyclic.
 Now combining those two resolutions of $\cV_1$, $\cV_3$, we have a natural chain map of exact sequences
\begin{eqnarray}\label{res:V_3} \xymatrix{
0 \ar[r] &  \Sym^{\star} \cV _1 \ar[r]_{\Sym^{\star} h} \ar[d]_{f_{\cV_1}}  &  \Sym^{\star} \cA_1 \ar[r] \ar[d]_{f_{\cA_1}} & \mathrm{Coker}_1 \ar[d]_{ f_{\cB_1}|_{\Coker_1}  } \ar[r] & 0    \\
   0 \ar[r] &  \cV_3 \ar[r] &     \cA_{3}  \ar[r]      &   \cB_{3}  \ar[r] & 0 , } \end{eqnarray}
   where $f_{\cA_1}$ is the composition  $\Sym^{\star} \cA_1 \ra \cV'_3 \ra \cA_3$; 
   $f_{\cB _1}|_{\Coker _1}$ is determined by  $f_{\cA _1}$; and $\mathrm{Coker}_1$ is defined as the 
   the quotient $\Sym^{\star} \cA_1 / \Sym^{\star} \cV_1$.

\medskip
{\em Step 2.} Furthermore let us take $\cA_1$ such that the natural map $\pi^*\pi_* \cA_1 \ra \cA_1$ is surjective.
This in turn implies that $\pi^*\pi_* \cB_1 \ra \cB_1$ is surjective and $\RR ^1 \pi _* (\cA_1^{\ot i} \ot \cB_1^{\ot j}) = 0$ for $i+j > 0$.
Thus $\RR ^1 \pi _* (\Sym ^i \cA_1 \ot \Lambda ^j \cB_1^{\ot j}) = 0$ and
 $$ \pi_* \Sym^{\star} [\cA_1 \ra \cB_1]     \xrightarrow{\text{qisom}}        \RR \pi _* \Sym^{\star} \cV _1 .$$
The morphism \eqref{der:maps} in the derived category is realized as each individual natural map as below except the dashed arrow.
\[ \xymatrix{ \Sym^{\star} \pi_* \cA_1 \ar[r] \ar[d]^{nat} & \Sym^{\star - 1} \pi_* \cA_1 \ot \pi_* \cB_1 \ar[r] \ar[d]^{nat} 
& \Sym^{\star - 2} \pi_* \cA_1 \ot \Lambda ^2 \pi_* \cB_1 \ar[r] \ar[d]^{nat} &  \cdots \\
                      \pi_* \Sym ^{\star} \cA _1 \ar[r] \ar[d]^{=}& \pi_* (\Sym ^{\star -1} \cA_1 \ot  \cB _1 ) \ar[r]_{\pi_* \partial _2}  \ar@{-->}[d] 
                                            &        \pi_* (\Sym^{\star -2} \cA_1 \ot  \Lambda ^2 \cB _1 ) \ar[r] \ar[d] & \cdots    \\
                                     \pi_* \Sym^{\star} \cA _1 \ar[r] \ar[d]^{\pi_* f_{\cA _1}}   &    \pi_*( \Coker_1 = \Ker (\partial _2) ) \ar[d]^{\pi_*f_{\cB _1}|_{\Coker_1}}  \ar[r]  &   0 \ar[r] \ar[d]   & \cdots \\  
                            \pi_* \cA_{3}  \ar[r]      &  \pi_* \cB_{3}  \ar[r]  & 0 \ar[r]  & \cdots } \]  
                            where $\partial _2: \Sym^{\star - 1} \cA_1 \ot  \cB _1 \ra \Sym^{\star -2} \cA_1 \ot \Lambda ^2 \cB _1$ is the differential and
                            $\star - 1$, $\star -2$ denote  the range of $[ 0, d_f -1]$, $ [0, d_f -2]$, respectively.
                            In other words, the first two rows present a cochain map representative of $nat$ in \eqref{der:maps}, 
                            the second and third rows are quasi-isomorphic, and the last two rows present a cochain map representative of $\RR \pi_* f$ in \eqref{der:maps}.
                            Here $nat$ is the natural map followed by projection. 
                            Note that canonically $\Coker _1 \cong \Ker \partial _2$.
 
Now, we will find a realization of the dotted arrow locally.
 By taking $\pi_*f_{\cB _1}|_{\Coker_1} \circ nat $, we obtain a $\cO_{\fBo}$-homomorphism
 \[ \varphi _{B_1}|_{\pi_*\cV _1} : \Sym^{\star - 1} \pi _* \cV _1 \ot B_1 \to B_3 , \] 
 where $B_i := \pi _* \cB _i$. 
 We want to find a lift of $\varphi _{B_1}|_{\pi_*\cV} $
 \[     \varphi _{B_1} ^{loc}: \Sym^{\star - 1} A_1  \ot B_1 \to B_3 \]
 locally on $\fBo$. Here $A_i := \pi _* \cA_i$.
Note that $\pi_* \Ker (\partial  _2) \ra  \pi_* (\Sym ^{\star - 1} \cA_1 \ot  \cB_1 )$ is a locally split monomorphism since its cokernel is locally free. 
Hence locally there is a dotted arrow making a quasi-isomorphism between the middle two complexes.  

In summary, we found that {\em locally} on $\fB^\circ$ there is $\varphi ^{loc}_{B_1}$ fitting in a cochain realization of $\Sym^{\star} \RR \pi _* \cV_1 \ra \RR\pi_* \cV_3$, i.e., \eqref{der:maps}:
\begin{equation}\label{AboveDiagram} \xymatrix{ \Sym^{\star} A_1 \ar[r] \ar[d]_{\varphi _{A_1}} & \Sym^{\star - 1} A_1 \ot B_1 \ar[r] \ar[d]_{\varphi ^{loc}_{B_1}} & \Sym^{\star - 2} A_1 \ot \Lambda ^2 B_1 \ar[r] \ar[d] &  \cdots \\
                             A_3 \ar[r]      &  B_3 \ar[r] & 0 \ar[r] & \cdots } \end{equation}
where $\varphi _{A_1}$ is the global $\varphi _{A_1} := \pi_* f_{\cA _1}  \circ nat $ restricted to the local chart.

\bigskip

{\em Construction of $\phi _{A_1}$, $\bar{\phi} _{B_1}$.}
We take an open substack $U^{\ke}$ of $\tot A_1$ such that $U^{\ke}$ is a DM stack and 
$Q^{\ke}_{V_1} := Q^{\ke}_{g, k} ( V_1/\!\!/G , d)$ is naturally a closed substack of $U^{\ke}$.
 So far, we found a cochain map (or, local cochain map) representatives on $\fBo$.
In below, we will define $\phi_{A_1} \in H^0 (U^{\ke}  , A_3|_{U^{\ke}} )$ and $\bar{\phi}_{B_1} \in \Hom_{\cO_{Q^{\ke}_{V_1}} } (B_1|_{Q^{\ke}_{V_1}} , B_3|_{Q^{\ke}_{V_1}} )$ 
from $\varphi_{A_1}$ and $\varphi _{B_1}|_{\pi_*\cV _1}$, respectively.
In Section \ref{M_degen} - \ref{Setup2} they will play some role in perfect obstruction theories \eqref{Obst} and \eqref{LGObst}.

We first 
 recall that the notion of a tautological section $t_{-}$ is introduced in the beginning of Section \ref{taut} and the exponential map $\exp$ is introduced in \eqref{expmap}.
Hence we have sections $\exp^{\star} (t_{A_1}) \in H^0 ( U^{\ke},  \Sym ^{\star} A_1|_{U^{\ke}})$
and $\exp^{\star -1}  t_{V_1} \in H^0 (Q^{\ke}_{V_1},  \Sym ^{\star - 1} \pi_* \cV_1|_{Q^{\ke}_{V_1}})$,
where $t_{V_1} = t_{A_1} |_{Q^{\ke}_{V_1}}$ is a section of $\pi_* \cV_1|_{Q^{\ke}_{V_1}}$.
We define a section and a homomorphism
 \begin{align}  \phi _{A_1} & := \varphi _{A_1}|_{U^{\ke}} \circ \exp^{\star} (t_{A_1}) \in  H^0 (U^{\ke},  A_3|_{U^{\ke}} ) ;   \nonumber \\ 
                    \label{def:barphi_B}     \bar{\phi}  _{B_1}  & := \varphi _{B_1}|_{Q^{\ke}_{V_1}} \circ (\exp^{\star -1} (t_{V_1})\ot - ) : B_1|_{Q^{\ke}_{V_1}} \ra B_3|_{Q^{\ke}_{V_1}} .
      \end{align} 
      Define $\phi_{B_1}^{loc}$ by $\varphi ^{loc}_{B_1} (\exp ^{\star -1} a\ot - )$ for $a\in A_1$ at the local chart, extending $\bar{\phi}_{B_1}$.

Let $d_{A_i}: \pi_* \cA_i =:A_i \ra \pi_* \cB_i = :B_i$ denote the differential maps induced from the differential maps
$\cA_i \to \cB_i$. By abuse notation, $d_{A_i}$ will denote its pullback to $U^{\ke}$ or its local charts.
The following lemma will be used in Section \ref{M_degen} to show that 
the cosection descends to that of the absolute obstruction sheaf.

\begin{Lemma}\label{key:com}  The equality 
$d_{A_3}\circ \phi _{A_1}   - \phi _{B_1}^{loc} \circ d_{A_1} \circ t_{A_1} = 0$ holds as  local sections of $B_3|_{U^{\ke}}$. \end{Lemma}

\begin{proof}
The equality means that $d_{A_3} \circ \varphi _{A_1} (\exp^{\star} a) =  \varphi^{loc} _{B_1} (\exp^{\star -1} a  \ot d_{A_1}(a)) $ for every $a\in A_1$.
This  is nothing but the commutativity of the first square in diagram \eqref{AboveDiagram}.
\end{proof}

\subsubsection{Paring and Residue map} \label{Sect:Res}
From now on in this section we will also use notation that 
\[ P:= B_3^\vee \text{ and } Q := A_3. \] Hence
$[Q \xrightarrow{d_Q} P^{\vee}]$ represents $\RR \pi _* (\cV_3:= \cV_2^{\vee})$.
Here the notation $P$ and $Q$ are named after `$p$-fields' introduced in \cite{CL} ($(R^1\pi _* \cV_3 )^{\vee} = \pi _* (\cV_3 \ot \omega _{\fC}) = \{ \text{\lq$p$-fields'} \}$)
 and its partner letter q ($\pi_* \cV_3 = \pi_* \cV_2 ^{\vee} =\{ \text{\lq$q$-fields'} \}$). The pairing of $p$-fields and $q$-fields are given by the residue map.

The pairing which we will discuss below will play a role to extend the Residue map \eqref{ResPairing} to define a cosection for a degeneration of \eqref{LGObst} to \eqref{Obst} in Section \ref{M_degen}.

Note that  
$[P \xrightarrow{d_P} Q^\vee]$ represents $\RR \pi _* (\cV _2 \ot \omega _{\fC} )$. Here $d_P :=   - d_Q^\vee$ due to the shifting.
This yields the cochain map realization of Serre-Grothendieck-Verdier duality $\RR \pi _* (\cV _2 \ot \omega _{\fC})[1] \xrightarrow{\sim} \RR \cH om ( \RR \pi_* \cV_2 ^\vee , \cO _{\fBo})$, 
which in turn gives rise to a cochain map realization of $\RR \pi _* (\cV _2 \ot \omega _{\fC}) [1] \ot  \RR \pi_* \cV_2 ^\vee \xrightarrow{\RR\mathrm{es} }  \cO _{\fBo}$ as
 \[ \xymatrix{ \ar[d] P \ot Q \ar[r] & P^\vee \ot P \oplus Q^\vee \ot Q \ar[d]_{\text{sum of}}^{\text{pairings}}  \ar[r] & Q^\vee \ot P^\vee \ar[d] \\
                      0      \ar[r]  & \cO_{\fBo} \ar[r] & 0 . } \]
By taking the $0$-th cohomology-level map of $\RRes$ above, we note that the parings restricted to  $\Ker d_P$, $\Ker d_Q$ are the residue
parings, i.e., the following diagram commute
\begin{align}\label{Pairing:Res}  \xymatrix{ P^\vee \ot \Ker d_P \oplus Q^\vee \ot \Ker d_Q \ar[r]^{\ \ \ \ \ \ \ \ \ \ \ \text{sum of parings}} \ar[d] &  \cO _{\fBo} .   \\
                        \RR ^1 \pi_* \cV _2 ^\vee   \ot \RR ^0 \pi _* (\cV_2 \ot \omega _{\fC})  \oplus \RR ^1 \pi_* (\cV _2  \ot \omega _{\fC}) \ot \RR ^0 \pi _* \cV_2 ^\vee
                          \ar[ur]^{\ \ \ \ \ \ \ \ \Res} & } \end{align}

\subsubsection{Set-up 1}\label{M_degen}
In this section, we will construct a degeneration using the preparations in Section \ref{CochainReal} and \ref{Sect:Res} 
in order to apply Corollary \ref{CoVir:Ch} for both stable quasimap spaces and stable LG quasimap spaces at central and generic points, respectively.

We consider \[   U_{PQ}:= U^{\ke} \times _{\fBo} \tot P  \times _{\fBo} \tot Q \subset \tot A_1 \times_{\fBo} \tot P \times_{\fBo} \tot Q  .\] 
Let $$p: F : =(B_1 \oplus Q^\vee \oplus Q ) |_{\UPQ} \ra \UPQ$$ be a vector bundle on $\UPQ$, defined via the natural map $\UPQ \ra \fB^{\circ}$.  
Let $ \tilde{F} \ra \tilde{U}_{PQ}$ be the pullback of $F$ by the projection map  $\tilde{U}_{PQ} :=\UPQ\times \mathbb{A}^1 \ra \UPQ$.
Let  $\kl$ be the coordinate of $\mathbb{A}^1$, let
$\lan\ ,\ \ran _{P}$ denote the pairing $P^\vee|_{\tilde{U}_{PQ}} \ot_{\cO_{\tilde{U}_{PQ}}} P|_{\tilde{U}_{PQ}} \to \cO _{\tilde{U}_{PQ}}$, and let $\lan \ ,\  \ran _{Q}$ be
the similar pairing for $Q$.
We consider a section $\kb$ of $\tilde{F}$ and a cosection $\sigma$ of $\tilde{F}|_{Z(\beta)}$ defined by
\begin{align} 
  \kb  & := (d_{A_1} \circ t_{A_1} ,\kl  d_P\circ  t_P,  \phi _{A_1} - \kl t_Q) ; \nonumber      \\
 \sigma  & := \lan   \bar{\phi}_{B_1}\circ \pr_{B_1} , t_P  \ran _P  + \lan  \mathrm{id}_{Q^\vee} \circ \pr_{Q^\vee} ,  t_Q \ran_Q +  \lan - d_Q \circ \pr_{Q}  , t_P \ran _{P}      
             \nonumber \\
         & = \lan  \bar{\phi}_{B_1} \circ \pr_{B_1}  - d_Q \circ \pr_{Q} , t_P \ran _{P} + \lan \mathrm{id}_{Q^\vee} \circ \pr_{Q^\vee} , t_Q \ran _Q  . \label{def:alpha} \end{align}
Here we suppress various pullback notation: for example the first term in $\kb$,
$d_{A_1}$ is $d_{A_1}|_{\tilde{F}}$. 
Let $\tilde{p}: \tot \tilde{F}_{|_{Z(\beta)}} \ra Z(\beta) $. Then we have the following.

 \begin{Cor}\label{van}
 $\tilde{p}^* \sigma \circ t_{\tilde{F}} = 0$ on the stack $C_{\Zb} \tilde{U}_{PQ}$.
 \end{Cor}

\begin{proof}
Consider the local extension $\ka ^{loc}$ of $\sigma$  by replacing $ \bar{\phi}_{B_1} $ with $\phi_{B_1}^{loc}$.
Note that $\ka ^{loc}\circ \beta = \lan \phi ^{loc}_{B_1} \circ d_{A_1} \circ t_{A_1}- d_Q \circ \phi _{A_1} , t_P \ran _P  $  which is zero by Lemma \ref{key:com}. 
Applying  $\ka ^{loc}\circ \beta =0$ to the deformation to the normal cone $C_{\Zb} \tilde{U}_{PQ}$ as in the proof of Lemma \ref{taut:lemma}, we conclude the proof.
\end{proof}

We consider the composite $d\beta$ of maps  $\cT_{\tilde{U}_{PQ}/\fB} |_{\Zb} \ra \cT_{\tilde{F} /\fB}|_{\Zb} \to \tilde{F}|_{\Zb}$ of 
the differential of $\beta$ relative to $\fB$ and the natural projection. Similarly we have
the composite $d_{{\bf k}}\beta : \cT_{\tilde{U}_{PQ}/{\bf k}} |_{\Zb} \ra \cT_{\tilde{F} /{\bf k}}|_{\Zb} \to \tilde{F}|_{\Zb}$ using the differential 
of $\beta$ relative to ${\bf k}$ and the natural projection. 
In the above proof we 
have shown that $\ka ^{loc}\circ \beta  = 0$.
Applying the chain rule to $\ka ^{loc} \circ \beta = 0$, we note that  $\sigma \circ d_{{\bf k}}\beta = 0$. Therefore
$\sigma$ restricted to $\tilde{F}|_{\Zb}$ is factored by some $\bar{\sigma}, \tilde{\sigma}$ as in a following commuting diagram:
 \begin{align}\label{Res:rest}   \xymatrix{   \cT_{\tilde{U}_{PQ}/\fB} |_{\Zb}  \ar[r]^{d\beta}  \ar[d]  &  \tilde{F}|_{\Zb}  \ar[r]^{}  \ar[d]_{=}
        &     \mathrm{Coker} \, d\beta \ar@/^3pc/[dd]^{\bar{\sigma}}       \ar[d]
             \\ 
              \cT_{\tilde{U}_{PQ}/{\bf k}} |_{\Zb}  \ar[r]^{d_{{\bf k}}\beta}   &  \tilde{F}|_{\Zb}  \ar[r]^{}  \ar[dr]_{\sigma}
        &     \mathrm{Coker} \, d_{{\bf k}}\beta \ar[d]^{\tilde{\sigma}}       \\
          &               &  \cO_{\Zb}      .    }   \end{align}

\subsubsection{The perfect obstruction theory for $Q_X^{\ke}$}\label{sect:surj} On the other hand, as shown \cite{CKW}, we can see that
the perfect obstruction theory 
 \eqref{Obst} has an explicit description on $Q_X^{\ke} := Q^{\ke}_{g, k} ( \Zs , d)$
as follows. 

First there is  a natural commuting diagram on the universal curve on $\fB^{\circ}$
\[ \xymatrix{\Sym^{\star - 1 } \cV_1 \ot \cV_1 \ar[d]_{prod}  \ar[r] &  \Sym^{\star - 1} \cV_1 \ot \cA_1 \ar[d]_{prod}  \ar[r]^{\mathrm{id}\ot d_{\cA_1}} &   \Sym^{\star - 1} \cV_1 \ot \cB_1 \ar[d]  \\ 
               \Sym^{\star} \cV_1 \ar[r] \ar[d]_{f_{\cV_1}}  &           \Sym^{\star} \cA_1 \ar[r]  \ar[d]_{f_{\cA_1}}  &   \Coker _1  \ar[d]_{f_{\cB_1}|_{\Coker _1}} \\ 
              \cV_3 \ar[r] &             \cA_3  \ar[r]_{d_{\cA_3}} & \cB _3  
               } \]
              where $\star -1$ ranges from $0$ to $d_f -1$ and $prod$ is induced from the quotient maps from tensor products to symmetric products.
Note that the compositions gives rise to a $\pi$-acyclic resolution of ${\bf u} ^* ( \cP \times _G d\uf ^{\vee})$ as
\begin{align}\label{K_dw}  \xymatrix{ 0 \ar[r] & \cV_1 |_{\fC} \ar[d]_{{\bf u}^*( \cP \times _G d\uf ^{\vee}) } \ar[r] & \cA_1 |_{\fC} \ar[d] \ar[r] & \cB_1|_{\fC} \ar[d] \ar[r] & 0  \\
    0 \ar[r] & \cV_3|_{\fC} \ar[r] & \cA_3|_{\fC} \ar[r] & \cB_3|_{\fC}  \ar[r] & 0  } \end{align}
          where $\fC$ is the universal curve on $Q^{\ke}_X$.  

Hence,  \eqref{Obst} is representable by the three-term complex at amplitude $[0, 1, 2]$
\[ A_1|_{Q_X^{\ke} } \xrightarrow{(d_{A_1} , \pi_*(f_{\cA _1} \circ prod ) (\exp t_{V_1} \ot -))|_{Q_X^{\ke}} } B_1|_{Q_X^{\ke} } \oplus Q|_{Q_X^{\ke} }  
\xrightarrow{ (\bar{\phi} _{B_1} - d_Q)|_{Q_X^{\ke}}} P^\vee|_{Q_X^{\ke} } . \] 
By Theorem 4.5.2 of \cite{CKM}, we know that  \eqref{Obst} is the dual of a perfect obstruction theory 
and hence the map $(\bar{\phi} _{B_1} - d_Q)|_{Q_X^{\ke}}$ above is surjective.
Let $K$ be the kernel of $(\bar{\phi} _{B_1} - d_Q)|_{Q_X^{\ke}}$.   
The three term complex is is quasi-isomorphic to $[A_1|_{Q_X^{\ke} }  \ra K ]$, which yields a concrete realization of the relative
obstruction theory for $Q_X^{\ke}$. Thus
\begin{equation}\label{vir exp of  QeX} [Q^{\ke}_X]^{\vir} = 0^!_K[C_{Q^{\ke}_X} U^{\ke}] , \end{equation}
where the inclusion $C_{Q^{\ke}_X} U^{\ke} \subset K$ is obtained by the section $(d_{A_1} \circ t_{A_1}, \phi _{A_1})$.

\subsubsection{Set-up 2} \label{Setup2}

We consider the Koszul complex $\{ \tilde{p}^* \sigma, t_{\tilde{F}}\}$ on $C_{\Zb} \tilde{U}_{PQ}$. Recall the notation $P=B_3^{\vee}$ and $Q=A_3$ made in Section
\ref{Sect:Res}.

\begin{Lemma}
The Koszul complex $\{ \tilde{p}^* \sigma, t_{\tilde{F}}\}$ is strictly exact off the zero locus of $\sigma$ and $\kb$, which is $Q^{\ke}_X  \times \mathbb{A}^1$.
\end{Lemma} 
\begin{proof} The first part is obvious by Lemma \ref{lem: strict off} (2) since the locus of $t_{\tilde{F}}=0$ in the cone is the vertex $Z(\beta )$ of the cone. 
It remains to explain why $Z(\beta ) \cap Z(\sigma)$ coincides with $Q^{\ke}_X  \times \mathbb{A}^1$. This easily follows as: when $\beta = \sigma = 0$,
we have
$t_Q =0$ by the second term of $\sigma$ in \eqref{def:alpha}; $\phi _{A_1}  =0$  by  the third term of $\kb$ and $t_Q=0$; 
$t_P=0$ by the first term in $\sigma$ in \eqref{def:alpha} and the surjectivity of $\bar{\phi} _{B_1} - t_Q$ over $Q^{\ke}_X$ (see \S \ref{sect:surj} for the surjectivity).
The converse is obvious so that the locus of $\beta = \sigma = 0$ coincides with the locus $Q^{\ke}_X$.
\end{proof}

For $\kl \in \mathbb{A}^1$, let $j_{\kl}: C_{Z(\beta _{\kl})} \UPQ \subset (C_{Z(\beta)} \tilde{U}_{PQ}) |_{\kl}$ be the natural closed immersion.
Recall $\tilde{p}: \tot \tilde{F}_{|_{Z(\beta)}} \ra Z(\beta)$ denotes the projection and  let $p_{\kl}: \tot F_{|_{Z(\beta _{\kl} )}} \ra Z(\beta_{\kl} )$
be projections. 
Now we have 
\begin{align} 
         & \kl ^! \ch ^{C_{Z(\beta)} \tilde{U}_{PQ}}_{Q^{\ke}_X \times \mathbb{A}^1} \{ \tilde{p}^* \sigma, t_{\tilde{F}}\} \cap [C_{Z(\beta)} \tilde{U}_{PQ} ]   \nonumber \\
          =  &    \ch ^{C_{Z(\beta)} \tilde{U}_{PQ} |_{\kl}}_{Q^{\ke}_X }  \{ p^*_{\kl} \sigma, t_F\} \cap \kl^![C_{Z(\beta)} \tilde{U}_{PQ} ]  \nonumber  \\ 
          =  &   \ch ^{C_{Z(\beta)} \tilde{U}_{PQ}|_{\kl}}_{Q^{\ke}_X} \{ p^*_{\kl} \sigma, t_F\} \cap j_{\kl *} [C_{Z(\beta _\kl )} \UPQ]              \nonumber  \\
          = &           \ch _{Q^{\ke}_X}^{C_{Z(\beta _{\kl})} \UPQ} \{ p^*_{\kl} \sigma, t_F\} \cap [C_{Z(\beta _{\kl})} \UPQ]  .  \label{Vir:Exp} 
         \end{align}
Here the first equality follows by the fact that $\ch _{Q^{\ke}(X) \times \mathbb{A}^1} ^{C_{Z(\beta)} \tilde{U}_{PQ}} \{ \tilde{p}^* \sigma , t_{\tilde{F}}\}$ is
a bivariant class so that it commutes with the refined Gysin homomorphism $\kl ^!$;
the second equality follows from Lemma 3.6 of \cite{CKW}; and
the third equality follows from the compatibility with proper pushforward.

\subsubsection{Proof of Theorem \ref{Comp:Thm}}
We prove  Theorem \ref{Comp:Thm}, i.e., Conjecture \ref{conj} by showing that \eqref{Vir:Exp} for $\lambda =0, 1$ 
becomes LHS and RHS of \eqref{Vir:Eq}, respectively, up to a common invertible factor.


Case $\lambda =1$: Let $U_P:= U^{\ke} \times _{\fBo} \tot P$.
First note that $Z(\beta |_{\kl =1} ) \xrightarrow{\cong} LGQ'$ by a projection. 
Under this isomorphism,
$C_{Z(\beta |_{\kl =1})} \UPQ \cong  C_{LGQ'} U_P\ti _{\fBo} Q$ as cones over $LGQ'$.
Let $\cF^1 := (B_1 \oplus Q^\vee )|_{LGQ'}$ and let $p: \cF^1 \to LGQ'$ be the projection.
Then $$\sigma |_{\kl =1, LGQ'}: (B_1\oplus Q^{\vee}\oplus Q)|_{LGQ'} \to \cO_{LGQ'}$$ coincides with 
$dw_{LGQ'} \oplus 0$ by  \eqref{dwq: 2}, 
 \eqref{def:barphi_B}, and  \eqref{Pairing:Res}. 

Therefore, 
we have \begin{align*}  \eqref{Vir:Exp}|_{\lambda =1}  
& = \ch^{C_{LGQ'} U_P\ti _{\fBo} Q}_{Q^{\ke}_X} ( \{ p^*dw_{LGQ'}, t_{\cF ^1}\}\boxtimes _{\fBo} \{ 0, t_Q \} ) [C_{LGQ'} U_P\ti _{\fBo} Q ] \\
                                     & = (\td Q|_{Q_X^{\ke}})^{-1}  \ch^{C_{LGQ'} LGU^{\ke}}_{Q^{\ke}_X}  ( \{ p^*dw_{LGQ'}, t_{\cF ^1}\} ) \cap [C_{LGQ'} U_P]   \\
                                     & = (\td F|_{Q_X^{\ke}})^{-1} \cap [LGQ']^{\vir}_{dw_{LGQ'}} . \end{align*} 
Here the first equality is explained above, the second equality is by Proposition 2.3 (vi) of \cite{PV:A}, and the third equality is explained in (just before
Conjecture \ref{conj}) \S \ref{LG:Moduli}.

\medskip

Hence, to complete the proof, 
it is enough to show that \eqref{Vir:Exp} for $\kl =0$ after multiplication by $ (-1)^{\chi (\cV ^{\vee}_2 )} \td F$  becomes LHS of \eqref{Vir:Eq}.

\medskip

Case $\kl=0$:  Let $\beta _0 := \beta |_{\lambda = 0}$. We have 
\begin{align}   Z (\beta _0)  & = Q_X^{\ke} \times_{\fBo} \tot P \times_{\fBo} \tot Q \nonumber \\
C_{Z(\beta _{0})} U_{PQ} & =  C_{Q^{\ke}_X} U^{\ke} \times_{\fBo} \tot P \times_{\fBo} \tot Q \subset F |_{Z(\beta _0)} .      \label{inc} \end{align}
Note that there is no constraint by $\beta _0$ on the part $\tot P\times_{\fBo} \tot Q$. 

By Lemma \ref{key:com}, 
 the inclusion \eqref{inc} is the composition of
\begin{equation}\label{eqn: C K F}  C_{Q^{\ke}_X} U^{\ke} \times_{\fBo} \tot P \times_{\fBo} \tot Q \subset q_0^*K \subset F|_{Z(\beta _0)} \end{equation} where 
$q_0: Z (\beta _0) \ra Q_X^{\ke}$ is the projection.

Let $m: C_{Z(\beta _{0})} \UPQ \ra C_{Q_X^{\ke}}U^{\ke}$ and $p_0: C_{Z(\beta _{0})} \UPQ \ra Z (\beta _0) $  be projections.
On $C_{Z(\beta _{0})} \UPQ $, we obtain a commuting diagram of locally free sheaves
\[ \xymatrix{   & \cO_{C_{Z(\beta _{0})} \UPQ } \ar[ld]_{m^* t_K}  \ar[d]^{t_F}  & & & \\
 p_0^*q_0^* K \ar[r] &    p^*_0(B_1\oplus Q \oplus Q^\vee)  \ar[rrr]^{(\bar{\phi} _{B_1} - d_Q , \mathrm{id}_{ Q^{\vee}  } ) }  \ar[d]^{p_0^*\sigma} &  & &   p_0^*(P^\vee \oplus Q^\vee) \ar[dlll]^{p_0^*taut} \\
             &               \cO_{C_{Z(\beta _{0})} \UPQ }  & & & } \]
where $taut$ is the sum of the tautological paring of dual pairs $(\tot P, P^\vee)$ and  $(\tot Q, Q^\vee)$.
Note that the commutativity of the first (resp. second) triangle above follows by \eqref{eqn: C K F}  (resp. \eqref{def:alpha}). 
By applying the Conclusion in Section \ref{deformation: spl} to the above case, 
on  $C_{Z(\beta _{0})} \UPQ $ we have a deformation of the complex $\{ p_0^*\sigma , t_F \}$ 
supported on $Q^{\ke}_X$ to $\{0, m^*t_K \}\ot \{ p_0^*taut, 0 \}$ supported also on $Q^{\ke}_X$.
By this deformation, $\eqref{Vir:Exp}|_{\kl=0}$ becomes the following:
\begin{align*} &    \ch^{ C_{Z(\beta _{0})} U_{PQ} }_{Q^{\ke}_X } (\{0, m^*t_K \}\ot \{ p_0^*taut, 0 \}) \cap [C_{Z(\beta _{0})} \UPQ] \\
  = &  \ch^{ C_{Z(\beta _{0})} \UPQ }_{Q^{\ke}_X } (\{0, m^*t_K \}\ot \Lambda ^{\bullet} (P^{\vee} \oplus Q^{\vee}) \ot \Lambda ^{top} (P \oplus Q)[top]) \cap [C_{Z(\beta _{0})} \UPQ] \\
 = &  (-1)^{\chi (\cV ^{\vee}_2 )}  \td (P \oplus Q)|_{Q_X^{\ke}} ^{-1} \cdot \ch (\Lambda ^{top} (P \oplus Q)|_{Q_X^{\ke}})  \cdot \ch^{C_{Q^{\ke}_X} U^{\ke} }_{Q^{\ke}_X} \{0, t_K \} \cap  [i]  ( [C_{Z(\beta _{0})} \UPQ] ) \\
= & (-1)^{\chi (\cV ^{\vee}_2 )}  (\td F |_{Q_X^{\ke}} )^{-1} \cdot \td K  \cdot \ch^{C_{Q^{\ke}_X} U^{\ke} }_{Q^{\ke}_X} \{0, t_K \} \cap  [i]  ( [C_{Z(\beta _{0})} \UPQ ] ) \\
  = &  (-1)^{\chi ( \cV ^{\vee}_2 )}  (\td F|_{Q_X^{\ke}})^{-1} \cdot \td K \cdot  \ch^{C_{Q^{\ke}_X} U^{\ke} }_{Q^{\ke}_X} \{0, t_K \} \cap [ C_{Q^{\ke}_X} U^{\ke}]  \\
  = & (-1)^{\chi ( \cV ^{\vee}_2 )}   (\td F|_{Q_X^{\ke}})^{-1} \cdot [Q^{\ke}_X]^{\vir} ,
 \end{align*}
where $\Lambda ^{\bullet} (P^{\vee} \oplus Q^{\vee})$ is the 2-periodic Koszul-Thom complex (see \cite[Proposition 2.3 (vi)]{PV:A}), 
$top$ denotes the rank of $P\oplus Q$, and $[i]$ is the canonical orientation of the inclusion $i : C_{Q^{\ke}_X} U^{\ke} \subset C_{Z(\beta _{0})} U_{PQ}$.
The first equality is from \eqref{duality}.
The second equality is from \cite[Proposition 2.3 (vi)]{PV:A}.
The third equality is from the easy fact that $\td E = \td E^\vee \cdot (\ch \, \mathrm{det} E^\vee )^{-1}$ for a vector bundle $E$.
The fourth  equality is from that $[i]  ( [C_{Z(\beta _{0})} \UPQ ] ) =  [ C_{Q^{\ke}_X} U^{\ke}] $.
The last equality is from Corollary \ref{Gysin} and \eqref{vir exp of QeX}.

\end{document}